\documentclass{amsart}
\usepackage{amssymb}
\usepackage{verbatim}
\usepackage{array}
\usepackage{latexsym}
\usepackage{enumerate}
\usepackage{tikz}
\usepackage{float}
\usepackage{amsmath}
\usepackage{amsfonts}
\usepackage{amsthm}
\usepackage{color}
\usepackage[english]{babel}

\newtheorem{theorem}{Theorem}[section]
\newtheorem{proposition}[theorem]{Proposition}
\newtheorem{lemma}[theorem]{Lemma}

\newtheorem{remark}[theorem]{Remark}

\def\F{\mathbb F}

\def\cH{\mathcal H}

\def\cM{\mathcal M}

\def\cX{\mathcal X}

\def\fqs{\mathbb{F}_{q^2}}

\def\PG{{\rm{PG}}}

\def\SL{{\rm{SL}}}
\def\ord{\mbox{\rm ord}}

\def\det{\mbox{\rm det}}

\def\min{{\rm min}}

\def\bq{{\bar q}}

\newcommand{\PSL}{\mbox{\rm PSL}}
\newcommand{\PGL}{\mbox{\rm PGL}}

\newcommand{\PSU}{\mbox{\rm PSU}}
\newcommand{\PGU}{\mbox{\rm PGU}}

\newcommand{\aut}{\mbox{\rm Aut}}

\newcommand{\diag}{\mbox{\rm diag}}

\title{The complete list of genera of quotients of the $\mathbb{F}_{q^2}$-maximal Hermitian curve for $q\equiv1\pmod4$}
\date{}
\author{Maria Montanucci and Giovanni Zini}

\begin{document}



\begin{abstract}

Let $\mathbb{F}_{q^2}$ be the finite field with $q^2$ elements.
Most of the known $\mathbb{F}_{q^2}$-maximal curves arise as quotient curves of the $\mathbb{F}_{q^2}$-maximal Hermitian curve $\mathcal{H}_{q}$.
After a seminal paper by Garcia, Stichtenoth and Xing \cite{GSX}, many papers have provided genera of quotients of $\mathcal{H}_q$, but their complete determination is a challenging open problem.
In this paper we determine completely the spectrum of genera of quotients of $\mathcal{H}_q$ for any $q\equiv1\pmod4$.

\end{abstract}

\maketitle

\begin{small}



\footnote{
This research was partially supported by Ministry for Education, University
and Research of Italy (MIUR)
 and by the Italian National Group for Algebraic and Geometric Structures
and their Applications (GNSAGA - INdAM).

URL: Maria Montanucci (maria.montanucci@unibas.it), Giovanni Zini (giovanni.zini@unimib.it).
}

\end{small}

\section{Introduction}
Let $q=p^n$ be a prime power, $\mathbb{F}_{q^2}$ the finite field with $q^2$ elements, and $\cX$ be a projective, absolutely irreducible, non-singular algebraic curve of genus $g$ defined over $\mathbb{F}_{q^2}$.
The curve $\cX$ is called $\mathbb{F}_{q^2}$-maximal if the number $|\cX(\mathbb{F}_{q^2})|$ of its $\mathbb{F}_{q^2}$-rational points attains the Hasse-Weil upper bound $q^2+1+2gq$. Maximal curves have been investigated for their applications in Coding Theory. Surveys on maximal curves are found in  \cite{FT,G,G2,GS,vdG,vdG2} and \cite[Chapter 10]{HKT}.

A well-known and important example of an $\mathbb{F}_{q^2}$-maximal curve is the Hermitian curve $\cH_q$. It is defined as any $\mathbb{F}_{q^2}$-rational curve which is projectively equivalent to the plane curve $Y^{q+1}-X^{q+1}+Z^{q+1}=0$. For fixed $q$, the curve $\cH_q$ has the largest genus $g(\cH_q)=q(q-1)/2$ that an $\mathbb{F}_{q^2}$-maximal curve can have. The full automorphism group $\aut(\cH_q)$ is isomorphic to $\PGU(3,q)$, the group of projectivities of ${\rm PG}(2,q^2)$ commuting with the unitary polarity associated with $\cH_q$. The automorphism group $\aut(\cH_q)$ is extremely large with respect to the value $g(\cH_q)$. Indeed it is know that the Hermitian curve is the unique curve of  genus $g \geq 2$ up to isomorphisms admitting an automorphism group of order at least equal to $16g^4$. 

By a result commonly referred to as the Kleiman-Serre covering result, see \cite{KS} and \cite[Proposition 6]{L}, a curve   $\cX$ defined over $\mathbb{F}_{q^2}$ which is $\mathbb{F}_{q^2}$-covered  by an $\mathbb{F}_{q^2}$-maximal curve is $\mathbb{F}_{q^2}$-maximal as well. In particular, $\mathbb{F}_{q^2}$-maximal curves can be obtained as Galois $\mathbb{F}_{q^2}$-subcovers of an $\mathbb{F}_{q^2}$-maximal curve $\cX$, that is, as quotient curves $\cX/G$ where $G$ is a finite automorphism group of $\cX$. Most of the known maximal curves are Galois covered by the Hermitian curve; see for instance \cite{GSX,CKT2,GHKT2,MZ,MZ2,MZNotFixing} and the references therein.

The first example of a maximal curve not Galois covered by the Hermitian curve is due to Garcia and Stichtenoth \cite{GS3}. This curve is $\mathbb{F}_{3^6}$-maximal and it is not Galois covered by $\cH_{27}$. It is a special case of the $\mathbb{F}_{q^6}$-maximal GS curve, which was later shown not to be Galois covered by $\cH_{q^3}$ for any $q>3$, \cite{GMZ,Mak}. Giulietti and Korchm\'aros \cite{GK} provided an $\mathbb{F}_{q^6}$-maximal curve, nowadays referred to as the GK curve, which is not covered by the Hermitian curve $\cH_{q^3}$ for any $q>2$. Two generalizations of the GK curve were introduced by Garcia, G\"uneri and Stichtenoth \cite{GGS} and by Beelen and Montanucci \cite{BM}. Both these generalizations are $\mathbb{F}_{q^{2n}}$-maximal curves, for any $q$ and odd $n \geq 3$. Also, they are not Galois covered by the Hermitian curve $\cH_{q^n}$ for $q>2$ and $n \geq 5$, see \cite{DM,BM}; the Garcia-G\"uneri-Stichtenoth's generalization is also not Galois covered by $\cH_{2^n}$ for $q=2$, see \cite{GMZ}.

A challenging open problem is the determination of the spectrum $\Gamma(q^2)$ of genera of $\mathbb{F}_{q^2}$-maximal curves, for given $q$. Apart from the examples listed above, most of the known values in $\Gamma(q^2)$ have been obtained from quotient curves $\cH_q/G$ of the Hermitian curve, which have been investigated in many papers.
The most significant cases are the following:
\begin{itemize}
\item $G$ fixes an $\mathbb{F}_{q^2}$-rational point of $\cH_q$; see \cite{BMXY,GSX,AQ}.
\item $G$ fixes a self-polar triangle in $\PG(2,q^2)\setminus\cH_q$; see \cite{DVMZ}.
\item $G$ normalizes a Singer subgroup of $\cH_q$ acting on three $\mathbb{F}_{q^6}$-rational points of $\cH_q$; see \cite{GSX, CKT}.
\item $G$ has prime order; see \cite{CKT2}.
\item $G$ fixes neither points nor triangles in ${\rm PG}(2,q^6)$; see \cite{MZNotFixing}. 
\end{itemize}

From the results already obtained in the literature (see \cite{DVMZ,MZNotFixing} and the references therein), in order to obtain the complete list of genera of quotients $\cH_q/G$, $G\leq\aut(\cH_q)$, only the following cases still have to be analyzed:
\begin{enumerate}
\item $G$ fixes an $\mathbb{F}_{q^2}$-rational point $P\notin\cH_q$, $p>2$.
\item $G$ fixes a point $P \in \cH_q(\mathbb{F}_{q^2})$, $p=2$ and $|G|=p^\ell d$ where $p^\ell \leq q$ and $d \mid (q-1)$.
\end{enumerate}
In this paper a complete analysis of Case 1 is given provided that $q$ is congruent to $1$ modulo $4$. This provides the complete list of genera of quotient of the Hermitian curve under this assumption.

More precisely, this paper is organized as follows. Section \ref{sec:preliminari} provides a collection of necessary preliminary results on the Hermitian curve and its automorphism group. 
In Section \ref{sec: ultimo} a complete analysis of Case 1 is given for $q \equiv 1 \pmod 4$. 
Section \ref{sec: lista} contains the complete list of genera of quotients of the Hermitian curve for $q \equiv 1 \pmod 4$ joining our results with the ones already obtained in the literature.

\section{Preliminary results}\label{sec:preliminari}

Throughout this paper $q=p^n$, where $p$ is a prime number and $n$ is a positive integer. The Deligne-Lusztig curves defined over a finite field $\mathbb{F}_q$ were originally introduced in \cite{DL}. Other than the projective line, there are three families of Deligne-Lusztig curves, named Hermitian curves, Suzuki curves and Ree curves. The Hermitian curve $\mathcal H_q$ arises from the algebraic group $^2A_2(q)={\rm PGU}(3,q)$ of order $(q^3+1)q^3(q^2-1)$. It has genus $q(q-1)/2$ and is $\mathbb F_{q^2}$-maximal. This curve is $\mathbb{F}_{q^2}$-isomorphic to the following curves:
\begin{equation} \label{M1}
X^{q+1}-Y^{q+1}-Z^{q+1}=0;
\end{equation}
\begin{equation} \label{M2}
X^{q}Z+XZ^{q}-Y^{q+1}=0.
\end{equation}
The automorphism group $\aut(\cH_q)$ is isomorphic to the projective unitary group $\PGU(3,q)$, and it acts on the set $\cH_q(\mathbb{F}_{q^2})$ of all $\mathbb{F}_{q^2}$-rational points of $\cH_q$ as $\PGU(3,q)$ in its usual $2$-transitive permutation representation.
The combinatorial properties of $\cH_q(\mathbb{F}_{q^2})$ can be found in \cite{HP}. The size of $\cH_q(\mathbb{F}_{q^2})$ is equal to $q^3+1$, and a line of $PG(2,q^2)$ has either $1$ or $q+1$ common points with $\cH_q(\mathbb{F}_{q^2})$, that is, it is either a tangent line or a chord of $\cH_q(\mathbb{F}_{q^2})$. Furthermore, a unitary polarity is associated with $\cH_q(\mathbb{F}_{q^2})$ whose isotropic points are those of $\cH_q(\mathbb{F}_{q^2})$ and isotropic lines are the tangent lines to $\cH_q(\mathbb{F}_{q^2})$, that is, the tangents to $\cH_q$ at the points of $\cH_q(\mathbb{F}_{q^2})$.

The following classification of subgroups of $\PGU(3,q)$ goes back to Mitchell \cite{M} and Hartley \cite{H}; see also \cite{MZNotFixing}.

\begin{theorem}\label{Mit} Let $G$ be a nontrivial subgroup of $\PGU(3,q)$. Then one of the following cases holds.
\begin{itemize}
\item[(i)] $G$ is contained in the maximal subgroup of $\PGU(3,q)$ of order $q^3(q^2-1)$ which stabilizes an $\F_{q^2}$-rational point of $\cH_q$.
\item[(ii)] $G$ is contained in the maximal subgroup $\cM_q$ of $\PGU(3,q)$ of order $q(q-1)(q+1)^2$ which stabilizes an $\F_{q^2}$-rational point of $\PG(2,q^2)\setminus\cH_q$; equivalently, $\cM_q$ stabilizes a chord of $\cH_q(\mathbb{F}_{q^2})$.
\item[(iii)] $G$ is contained in the maximal subgroup of $\PGU(3,q)$ of order $6(q+1)^2$ which stabilizes a self-polar triangle of $\PG(2,q^2)\setminus\cH_q$ with respect to the unitary polarity associated to $\cH_q(\mathbb{F}_{q^2})$.
\item[(iv)] $G$ is contained in the maximal subgroup of $\PGU(3,q)$ of order $3(q^2-q+1)$ which stabilizes a triangle in $\cH_{q}(\F_{q^6})\setminus\cH_{q}(\F_{q^2})$ invariant under the Frobenius collineation $\Phi_{q^2}:(X,Y,Z)\mapsto (X^{q^2},Y^{q^2},Z^{q^2})$ of $\PG(2,\bar{\mathbb{F}}_{q})$.
\item[(v)] $G$ stabilizes no points nor triangles in $\PG(2,\bar{\F}_{q^2})$.
\end{itemize}
\end{theorem}

We also recall the classification of subgroups of $\SL(2,q)$ and $\PGL(2,q)$.

\begin{theorem} \label{sl2q}{\rm (see \cite[Theorem 6.17]{Suzuki})}
Up to isomorphism, the complete list of subgroups of $\SL(2,q)$ is the following. \\
$\bullet$ Tame subgroups:
\begin{enumerate}
\item $\SL(2,5)$, when $p \geq 7$ and $q^2 \equiv 1\pmod 5$.
\item $\hat{\Sigma}_4$, the representation group of $S_4$ in which the transpositions corerspond to the elements of order $4$, when $p \geq 5$ and $q^2 \ \equiv 1 \pmod{16}$; $\hat{\Sigma}_4$ is isomorphic to $SmallGroup(48,28)$.
\item $\SL(2,3)$, when $p \geq 5$.
\item The cyclic group $C_d$ of order $d$, where $d \mid (q \pm 1)$.
\item The dicyclic group $Dic_m=\langle \delta,\epsilon \mid \delta^{2m}=1, \epsilon^2=\delta^m, \ \epsilon^{-1}\delta\epsilon=\delta^{-1}\rangle \cong C_{2m}\circ C_4$ of order $4m$, where $1<m \mid \frac{q \pm 1}{2}$.
\end{enumerate}
$\bullet$ Non-tame subgroups:
\begin{enumerate}
\item $E_{p^k} \rtimes C_d$, where $d \mid \gcd(p^k-1,q-1)$, $k\leq n$, and $E_{p^k}$ is elementary abelian of order $p^k$.
\item $\SL(2,5)$, when $p=3$ and $q^2\equiv1\pmod5$;
\item $\SL(2,p^k)$, where $k \mid n$;
\item ${\rm TL} (2,p^k) \cong \langle \SL(2,p^k), d_\pi \rangle$, where $k\mid n$, $n/k$ is even,
$$d_\pi = \begin{pmatrix}
    w & 0 & 0  \\
    0 & w^{-1} & 0 \\
    0 & 0 & 1
  \end{pmatrix},$$
$w=\xi^{\frac{p^k+1}{2}}$, and $\mathbb{F}^*_{p^{2k}}= \langle \xi \rangle$.
\end{enumerate}
\end{theorem}
\begin{theorem}\label{Di}{\rm (see \cite[Chapt. XII, Par. 260]{D}, \cite[Kap. II, Hauptsatz 8.27]{Hup}, \cite[Thm. A.8]{HKT})}
The following is the complete list of subgroups of $\PGL(2,q)$ up to conjugacy:
\begin{itemize}
\item[(i)] the cyclic group of order $h$ with $h\mid(q\pm1)$;
\item[(ii)] the elementary abelian $p$-group of order $p^f$ with $f\leq n$;
\item[(iii)] the dihedral group of order $2h$ with $h\mid(q\pm1)$;
\item[(iv)] the alternating group $\mathbf{A}_4$ for $p>2$, or $p=2$ and $n$ even;
\item[(v)] the symmetric group $\mathbf{S}_4$ for $16\mid(q^2-1)$;
\item[(vi)] the alternating group $\mathbf{A}_5$ for $p=5$ or $5\mid(q^2-1)$;
\item[(vii)] the semidirect product of an elementary abelian $p$-group of order $p^k$ by a cyclic group of order $h$, with $k\leq n$ and $h\mid\gcd(p^f-1,q-1)$;
\item[(viii)] $\PSL(2,p^f)$ for $f\mid n$;
\item[(ix)] $\PGL(2,p^f)$ for $f\mid n$.
\end{itemize}
\end{theorem}

In our investigation it is useful to know how an element of $\PGU(3,q)$ of a given order acts on $\PG(2,\bar{\mathbb{F}}_{q})$, and in particular on $\cH_q(\mathbb{F}_{q^2})$. This can be obtained as a corollary of Theorem \ref{Mit}, and is stated in Lemma $2.2$ with the usual terminology of collineations of projective planes; see \cite{HP}. In particular, a linear collineation $\sigma$ of $\PG(2,\bar{\mathbb{F}}_q)$ is a $(P,\ell)$-\emph{perspectivity}, if $\sigma$ preserves  each line through the point $P$ (the \emph{center} of $\sigma$), and fixes each point on the line $\ell$ (the \emph{axis} of $\sigma$). A $(P,\ell)$-perspectivity is either an \emph{elation} or a \emph{homology} according to $P\in \ell$ or $P\notin\ell$.
This classification result was obtained in \cite{MZ}.
\begin{lemma}\label{classificazione}
For a nontrivial element $\sigma\in\PGU(3,q)$, one of the following cases holds.
\begin{itemize}
\item[(A)] ${\rm ord}(\sigma)\mid(q+1)$ and $\sigma$ is a homology, with center $P\in\PG(2,q^2)\setminus\cH_q$ and axis $\ell$ which is a chord of $\cH_q(\mathbb{F}_{q^2})$; $(P,\ell)$ is a pole-polar pair with respect to the unitary polarity associated to $\cH_q(\mathbb{F}_{q^2})$.
\item[(B)] $p\nmid{\rm ord}(\sigma)$ and $\sigma$ fixes the vertices $P_1,P_2,P_3$ of a non-degenerate triangle $T\subset\PG(2,q^6)$.
\begin{itemize}
\item[(B1)] $\ord(\sigma)\mid(q+1)$, $P_1,P_2,P_3\in\PG(2,q^2)\setminus\cH_q$, and the triangle $T$ is self-polar with respect to the unitary polarity associated to $\cH_q(\mathbb{F}_{q^2})$.
\item[(B2)] $\ord(\sigma)\mid(q^2-1)$ and $\ord(\sigma)\nmid(q+1)$; $P_1\in\PG(2,q^2)\setminus\cH_q$ and $P_2,P_3\in\cH_q(\mathbb{F}_{q^2})$.
\item[(B3)] $\ord(\sigma)\mid(q^2-q+1)$ and $P_1,P_2,P_3\in\cH_q(\mathbb{F}_{q^6})\setminus\cH_q(\mathbb{F}_{q^2})$.
\end{itemize}
\item[(C)] ${\rm ord}(\sigma)=p$ and $\sigma$ is an elation with center $P\in\cH_q(\mathbb{F}_{q^2})$ and axis $\ell$ which is tangent to $\cH_q$ at $P$, such that $(P,\ell)$ is a pole-polar pair with respect to the unitary polarity associated to $\cH_q(\mathbb{F}_{q^2})$.
\item[(D)] ${\rm ord}(\sigma)=p$ with $p\ne2$, or ${\rm ord}(\sigma)=4$ and $p=2$; $\sigma$ fixes a point $P\in\cH_q(\mathbb{F}_{q^2})$ and a line $\ell$ which is tangent to $\cH_q$ at $P$, such that $(P,\ell)$ is a pole-polar pair with respect to the unitary polarity associated to $\cH_q(\mathbb{F}_{q^2})$.
\item[(E)] $\ord(\sigma)=p\cdot d$ where $d$ is a nontrivial divisor of $q+1$; $\sigma$ fixes two points $P\in\cH_q(\mathbb{F}_{q^2})$ and $Q\in\PG(2,q^2)\setminus\cH_q$, the polar line $PQ$ of $P$, and the polar line of $Q$ which is another line through $P$.
\end{itemize}
\end{lemma}

Throughout the paper, a nontrivial element of $\PGU(3,q)$ is said to be of type (A), (B), (B1), (B2), (B3), (C), (D), or (E), as given in Lemma \ref{classificazione}.

To compute the genus of a quotient curve we make use of the Riemann-Hurwitz formula; see \cite[Theorem 3.4.13]{Sti}. For any subgroup $G$ of $\PGU(3,q)$, the cover $\cH_q\rightarrow\cH_q/G$ is a Galois cover defined over $\mathbb{F}_{q^2}$ and the degree $\Delta$ of the different divisor is given by the Riemann-Hurwitz formula, namely
$\Delta=(2g(\cH_q)-2)-|G|(2g(\cH_q/G)-2)$. On the other hand, $\Delta=\sum_{\sigma\in G\setminus\{id\}}i(\sigma)$, where $i(\sigma)\geq0$ is given by the Hilbert's different formula \cite[Thm. 3.8.7]{Sti}, namely $\textstyle{i(\sigma)=\sum_{P\in\cH_q(\bar\F_q)}v_P(\sigma(t)-t)}$, where $t$ is a local parameter at $P$.


By analyzing the geometric properties of the elements $\sigma \in \PGU(3,q)$, it turns out that there are only a few possibilities for $i(\sigma)$.
This is obtained as a corollary of Lemma \ref{classificazione} and stated in the following theorem, which is proved in \cite{MZ}.

\begin{theorem}\label{caratteri}
For a nontrivial element $\sigma\in \PGU(3,q)$ one of the following cases occurs.
\begin{enumerate}
\item If $\ord(\sigma)=2$ and $2\mid(q+1)$, then $\sigma$ is of type {\rm(A)} and $i(\sigma)=q+1$.
\item If $\ord(\sigma)=3$, $3 \mid(q+1)$ and $\sigma$ is of type {\rm(B3)}, then $i(\sigma)=3$.
\item If $\ord(\sigma)\ne 2$, $\ord(\sigma)\mid(q+1)$ and $\sigma$ is of type {\rm(A)}, then $i(\sigma)=q+1$.
\item If $\ord(\sigma)\ne 2$, $\ord(\sigma)\mid(q+1)$ and $\sigma$ is of type {\rm(B1)}, then $i(\sigma)=0$.
\item If $\ord(\sigma)\mid(q^2-1)$ and $\ord(\sigma)\nmid(q+1)$, then $\sigma$ is of type {\rm(B2)} and $i(\sigma)=2$.
\item If $\ord(\sigma)\ne3$ and $\ord(\sigma)\mid(q^2-q+1)$, then $\sigma$ is of type {\rm(B3)} and $i(\sigma)=3$.
\item If $p=2$ and $\ord(\sigma)=4$, then $\sigma$ is of type {\rm(D)} and $i(\sigma)=2$.
\item If $\ord(\sigma)=p$, $p \ne2$ and $\sigma$ is of type {\rm(D)}, then $i(\sigma)=2$.
\item If $\ord(\sigma)=p$ and $\sigma$ is of type {\rm(C)}, then $i(\sigma)=q+2$.
\item If $\ord(\sigma)=p\cdot d$ with $1\ne d\mid(q+1)$, then $\sigma$ is of type {\rm(E)} and $i(\sigma)=1$.
\end{enumerate}
\end{theorem}

In order to characterize the genera of $\cH_q/G$ for any subgroup $G$ of $\PGU(3,q)$ under the assumption $q\equiv1\pmod4$, it is sufficient to consider the case $G\leq\cM_q$, where $\cM_q$ is the maximal subgroup {\it (ii)} in Theorem \ref{Mit}.
In fact, if $G\not\leq\cM_q$:
\begin{itemize}
\item the genera $g(\cH_q/G)$ for the subgroups $G\leq\PGU(3,q)$ stabilizing an $\mathbb{F}_{q^2}$-rational point of $\cH_q$ are characterized in \cite[Theorem 1.1]{BMXY};
\item the genera $g(\cH_q/G)$ for the subgroups $G\leq\PGU(3,q)$ stabilizing a self-polar triangle of $\PG(2,q^2)$ are characterized in \cite[Section 3]{DVMZ};
\item the genera $g(\cH_q/G)$ for the subgroups $G\leq\PGU(3,q)$ stabilizing a Frobenius-invariant triangle in $\cH_q(\F_{q^6})\setminus\cH_q(\F_{q^2})$ are characterized in \cite[Proposition 4.2]{CKT2};
\item the genera $g(\cH_q/G)$ for the subgroups $G\leq\PGU(3,q)$ which do not stabilize any point or triangle are characterized in \cite{MZNotFixing}.
\end{itemize}

\section{The maximal subgroup $\cM_q$ for $q\equiv1\pmod4$} \label{sec: ultimo}


Let $\cM_q$ be the maximal subgroup of $\PGU(3,q)$ stabilizing a point $P\in\PG(2,q^2)\setminus\cH_q$ and its polar line $\ell$, which is a chord of $\cH_q(\F_{q^2})$.
For any odd $q$, the structure of $\cM_q$ was already given in \cite[Section 3]{CKT2} and \cite[Section 3]{MZ2} as a semidirect product $\cM_q=\Gamma\rtimes C$, where $\Gamma$ is the commutator subgroup of $\cM_q$ and is isomorphic to $\SL(2,q)$, while $C$ is a cyclic group of order $q+1$; this description was used to compute the genera $g(\cH_q/G)$ for some $G\leq\cM_q$, namely when $G$ is a tame subgroup of $\Gamma$ in \cite{CKT2}, and when $G=(G\cap\Gamma)\rtimes(G\cap C)$ in \cite{MZ2}.

Henceforth, we assume $q\equiv1\pmod4$ and use a different description of $\cM_q$.
Let $\cH_q$ be given by the model \eqref{M1}. Up to conjugation in $\PGU(3,q)$, we can assume that $P=(0,0,1)$ and $\ell$ has equation $Z=0$. As pointed out in the proof of Theorem 2.5 in \cite{BM},
$$ \cM_q=\left\{\begin{pmatrix} a & \zeta c^q & 0 \\ c & \zeta a^q & 0 \\ 0 & 0 & 1 \end{pmatrix} \mid a,c\in\F_{q^2},\, a^{q+1}-c^{q+1}=1,\, \zeta^{q+1}=1\right\} $$
and the commutator subgroup of $\cM_q$ is
$$ H=\left\{\begin{pmatrix} a & c^q & 0 \\ c & a^q & 0 \\ 0 & 0 & 1 \end{pmatrix}\mid a,c\in\F_{q^2},\, a^{q+1}-c^{q+1}=1\right\}\cong\SL(2,q). $$
The center $Z$ of $\cM_q$ is cyclic of order $q+1$ and is made by the elements of type (A) with center $P$; see \cite[Section 3]{MZ2}.
Hence,
$$ Z=\left\{\begin{pmatrix} a & 0 & 0 \\ 0 & a & 0 \\ 0 & 0 & 1 \end{pmatrix}\mid a^{q+1}=1\right\}. $$
The intersection of $I=H\cap Z$ has order $2$ and is generated by the unique involution
$$\iota=\diag(-1,-1,1)$$
of $H$. Since $\frac{q+1}{2}$ is odd, the group generated by $H$ and $Z$ is a direct product $HZ=H\times \Omega$, where $\Omega\cong C_{\frac{q+1}{2}}$ is the subgroup of $Z$ of order $\frac{q+1}{2}$.

The group $HZ$ has index $2$ in $\cM_q$, and contains exactly one involution $\iota$. Let $\beta$ be any involution of $\cM_q$ different from $\iota$, for instance $\beta=\diag(-1,1,1)$; obviously, $\beta$ normalizes both $H=\cM_q^{\prime}$ and $Z=Z(\cM_q)$. Then
$$ \cM_q=(H\rtimes \langle\beta\rangle)\times \Omega \cong (\SL(2,q)\rtimes C_2)\times C_{\frac{q+1}{2}}\,. $$
With the notation of \cite[Section 9]{GKeven} we also denote $H\rtimes \langle\beta\rangle$ by ${\rm SU^{\pm}}(2,q)$, meaning that $H\rtimes\langle\beta\rangle$ consists of the elements of $\cM_q$ with determinant $1$ or $-1$; here, the determinant of an element $\alpha\in\cM_q$ is the determinant of the representative matrix of $\alpha$ having entry $1$ in the third row and column.
Note that $H\rtimes\langle\beta \rangle$ is conjugated to $H\rtimes\langle\gamma\rangle$ in $\cM_q$ for any involution $\gamma\in\cM_q\setminus\left\{\iota\right\}$, because $\gamma$ and $\beta$ are conjugated in $\cM_q$ (see \cite[Lemma 2.2]{KOS}) and $H$ is normal in $\cM_q$.
Note also that $\cM_q$ contains no elements of type (B3) or (D).

For any $G\leq\cM_q$, we will use the following notation:
$$G_{\pm}=G\cap {\rm SU}^{\pm}(2,q),\quad G_H=G\cap H,\quad G_\Omega=G\cap\Omega,\quad \omega=|G_\Omega|.$$

We determine in the following proposition the subgroups of ${\rm SU}^{\pm}(2,q)$.

\begin{proposition}\label{sottogruppiSU}
The following is the complete list of subgroups of ${\rm SU}^{\pm}(2,q) \leq \cM_q$ for $q\equiv1\pmod4$.
\begin{itemize}
\item The subgroups of $H$, listed in Theorem {\rm \ref{sl2q}}.
\item Cyclic groups of order $2$.
\item Groups $\SL(2,3)\rtimes C_2 \cong SmallGroup(48,29)$ when $p\geq5$ and $8\nmid(q-1)$.
\item Cyclic groups of order $2d>2$, where either $d\mid (q-1)$ and $d\nmid\frac{q-1}{2}$, or $d\mid(q+1)$.
\item Abelian groups $C_d\times C_2$, where $d\mid(q+1)$ and $d$ is even.
\item Dihedral groups of order $2d$, where $d\mid(q\pm1)$.
\item
 Groups $\hat{Dic}_m=\langle\alpha,\epsilon\mid\alpha^{4m}=1,\epsilon^2=\alpha^{2m},\epsilon^{-1}\alpha\epsilon=\alpha^{2m-1}\rangle$ of order $8m$ extending a subgroup $Dic_m$ of $H$, when $m\mid\frac{q-1}{2}$ and $m\nmid\frac{q-1}{4}$.
\item Groups $ Dic_m \rtimes C_2$ of order $8m$ extending a subgroup $Dic_m$ of $H$, when $1<m\mid\frac{q+1}{2}$.
\item Groups $E_{p^k}\rtimes C_{2d}$, where $k\leq n$, $d\mid\gcd(p^k-1,q-1)$, $d$ is even, $E_{p^k}$ is elementary abelian of order $p^k$, and $C_{2d}$ is cyclic of order $2d$.
\item Groups ${\rm SU}^{\pm}(2,p^k)\cong {\rm SL}(2,p^k)\rtimes C_2$, where $k\mid n$ and $n/k$ is odd.
\end{itemize}
\end{proposition}

\begin{proof}
Let $G\leq H\rtimes\langle\beta\rangle$ and assume $G\not\leq H$, so that $G_H=G\cap H$ has index $2$ in $G$. We may assume that $G$ has order greater than $2$, that is $G_H$ is a nontrivial subgroup of $H$.
If $\iota$ is the unique involution of $H$, we denote by $\bar{G}$ and $\bar{G}_H$ the images of $G$ and $G_H$ under the canonical epimorphism ${\rm SU}^{\pm}(2,q)\rightarrow {\rm SU}^{\pm}(2,q)/\langle\iota\rangle$.
Since $\iota$ is the kernel of the action of ${\rm SU}^{\pm}(2,q)$ on the line $\ell$, the action of ${\rm SU}^{\pm}(2,q)/\langle\iota\rangle$ on the $q+1$ points of $\cH_q\cap\ell$ is equivalent to the action of a subgroup of $\PGL(2,q)$;
as they have the same order, ${\rm SU}^{\pm}(2,q)/\langle\iota\rangle = \PGL(2,q)$.
Note that  $|\bar{G}|=|G_H|$ or $|\bar{G}|=2|\bar{G}|$ according to $|G_H|$ being even or odd, respectively. We classify $G_H$ according to Theorem \ref{sl2q}.
\begin{itemize}
\item Suppose $G_H=\SL(2,5)$. Then $\PGL(2,q)$ has a subgroup $\bar G$ of order $120$ and $\bar{G}$ contains $\bar{G}_H\cong \mathbf{A}_5$, a contradiction to Theorem \ref{Di}.
\item Suppose $G_H=\hat{\Sigma}_4$. As in the previous point, $\PGL(2,q)$ has a subgroup of order $48$ containing a subgroup isomorphic to $\mathbf{S}_4$, a contradiction to Theorem \ref{Di}.
\item Suppose $G_H=E_{p^k}\rtimes C_d$ with $k\leq n$ and $d\mid\gcd(p^k-1,q-1)$. The unique Sylow $p$-subgroup $E_{p^k}$ of $G$ is normal in $G$, and hence $G$ fixes the unique fixed point $Q\in\ell$ of $E_{p^k}$ on $\cH_q$; see \cite[Lemma 11.129]{HKT}.
If $d$ is odd, then $|G|=2|G_H|$ and $G\setminus G_H$ contains an involution $\omega$, which is of type (A) and has center on $\ell$; this is a contradiction, since $\omega$ cannot fix any point on $\ell\cap\cH_q$ by Lemma \ref{classificazione}.
Then $d$ is even. By \cite[Lemma 11.44]{HKT}, $G=E_{p^k}\rtimes C_{2d}$. Such a group $G=E_{p^k}\rtimes C_{2d}$ actually exists in $\cM_q$: for instance, as in \cite{GSX}, use the model \eqref{M2} for $\cH_q$, assume up to conjugacy that $P=(0,1,0)$, and define
$$ G=\left\{ \begin{pmatrix} 1 & 0 & c \\ 0 & 1 & 0 \\ 0 & 0 & 1 \end{pmatrix}\mid c^{p^k}+c=0 \right\} \rtimes \left\{\begin{pmatrix} a^{p^k+1} & 0 & 0 \\ 0 & a & 0 \\ 0 & 0 & 1 \end{pmatrix}\mid a^{2d}=1 \right\} \cong E_{p^k}\rtimes C_{2d}. $$
Then $G\leq\cM_q$. Since $\cM_q={\rm SU}^{\pm}(2,q)\times C_{\frac{q+1}{2}}$ and $\gcd\left(|G|,\frac{q+1}{2}\right)=1$, this implies $G\leq{\rm SU}^{\pm}(2,q)$.
\item Suppose $G_H=\SL(2,p^k)$ with $k\mid n$. Then $|\bar{G}|=|\PGL(2,p^k)|$ and $\bar G$ contains $\bar{G}_H \cong \PSL(2,p^k)$, so that $\bar{G}=\PGL(2,p^k)$ by Theorem \ref{Di}.
Let $G_1$ be a subgroup of ${\rm SU}^{\pm}(2,q)$ with $G_H\leq G_1$ and $[G_1:G_H]=2$. Clearly, $\bar{G}_1=\bar{G}=\PGL(2,p^k)$; we show that $G_1=G$.
Choose $\delta\in\bar{G}$ of order $o(\delta)=p^k-1$ if $4\mid(p^k-1)$, or $o(\delta)=p^k+1$ if $4\mid(p^k+1)$.
Let $\alpha\in G$ and $\bar{\alpha}_1 \in G_1$ be preimages of $\delta$, that is $\bar{\alpha}=\bar{\alpha}_1 =\delta$; their order is $o(\alpha)=o(\alpha_1)=2\cdot o(\delta)$ and divides $2(q-1)$.
This implies in particular $\alpha,\alpha_1\notin G_H$, so that $G=\langle G_H,\alpha\rangle$ and $G_1=\langle G_H,\alpha_1\rangle$.
The group $\PGU(3,q)$ contains a unique cyclic subgroup $C$ such that $\delta\in C$ and $[C:\langle\delta\rangle]=2$, see Lemma \ref{classificazione}; hence, $\langle\alpha\rangle=\langle\alpha_1\rangle$. Thus, $G_1=G$: there is at most one subgroup of ${\rm SU}^{\pm}(2,q)$ containing $G_H$ with index $2$.

If $n/k$ is even, then $G={\rm TL}(2,p^k)$ by Theorem \ref{sl2q}.
Assume that $n/k$ is odd. Then the group
\begin{equation}\label{sottoSUnelsottocampo}
G=\left\{\begin{pmatrix} a & c^q & 0 \\ c & a^q & 0 \\ 0 & 0 & 1 \end{pmatrix} \mid a,c\in\mathbb{F}_{p^{2k}},\, a^{p^k+1}-c^{p^k+1}=1\right\} \rtimes \left\{\begin{pmatrix} \zeta & 0 & 0 \\ 0 & 1 & 0 \\ 0 & 0 & 1 \end{pmatrix}\mid \zeta=\pm1\right\}
\end{equation}
is a subgroup of ${\rm SU}^{\pm}(2,q)$; this can be proved in analogy to the proof of \cite[Lemma 3.10]{MZNotFixing}, i.e. showing by induction on $n/k$ that the condition $a^{p^k+1}-c^{p^k+1}=1$ is equivalent to $a^{q+1}-c^{q+1}=1$ for any $a,c\in\mathbb{F}_{p^{2k}}$.
Clearly, $G\cong {\rm SU}^{\pm}(2,p^k)= SL(2,p^k)\rtimes C_2$.

\item Suppose $G_H={\rm TL}(2,p^k)$. Then $\PGL(2,q)$ has a subgroup $\bar G$ such that $|\bar G|=p^k(p^{2k}-1)$ and $\PGL(2,p^k)=\bar{G}_H\leq \bar{G}$, a contradiction to Theorem \ref{Di}.

\item Suppose $G_H=C_d=\langle\alpha\rangle$ cyclic of order $d\mid(q\pm1)$, so that $|G|=2d$.
Firstly, we prove that the conditions in the statement when $G$ is abelian or dihedral are necessary for the existence of $G$; secondly, we show that such conditions are also sufficient.

Note that, if $\alpha\in {\rm SU}^{\pm}(2,q)\setminus H$ and $o(\alpha)>2$, then $o(\alpha)\nmid(q-1)$.
In fact, if $2<o(\alpha)\mid(q-1)$, then $\alpha$ is of type (B2) and fixes exactly two points other than $P$, say $Q,R\in\ell\cap\cH_q$; but the pointwise stabilizer $S$ of $\{Q,R\}$ in $\PGU(3,q)$ is cyclic of order $q^2-1$ (see Lemma \ref{classificazione}), and $|S\cap H|=q-1$, which implies $\alpha\in H$.
Hence, if $d\mid(q-1)$ and $G$ is cyclic, then $d\nmid\frac{q-1}{2}$.

We can assume that $d>2$. A generator $\delta$ of $G_H$ is either of type (B1) or of type (B2) and has three fixed points $P,Q,R$, where $Q,R\in\ell$; since $G_H$ is normal in $G$, $G$ acts on $\{Q,R\}$.
Let $\gamma\in G\setminus G_H$.
If $\gamma(Q)=Q$, $\gamma(R)=R$, and $d\mid(q-1)$, then $G$ is cyclic because the pointwise stabilizer of $\{P,Q,R\}$ in $\PGU(3,q)$ is cyclic.
If $\gamma(Q)=Q$, $\gamma(R)=R$, and $d\mid(q+1)$, then $G$ is contained in the pointwise stabilizer $C_{q+1}\times C_{q+1}$ of $\{P,Q,R\}\subset\PG(2,q^2)\setminus\cH_q$, see \cite[Section 3]{DVMZ}; hence, $G=C_d\times C_2$ (which is $C_{2d}$ when $d$ is odd).
Assume that $\gamma$ interchanges $Q$ and $R$. Then $\gamma^2$ fixes $\ell$ pointwise, so that either $\gamma$ is an involution or $\gamma^2=\iota$, the unique element of type (A) in $H$.
If $\gamma^2=\iota$, then $\gamma$ has order $4\mid(q-1)$, and hence $\gamma\in G_H$, a contradiction.
Then $\gamma$ is an involution, and $G$ is a semidirect (eventually direct) product $G_H\rtimes \langle\gamma\rangle$; denote by $\tilde{Q},\tilde{R}\in\PG(2,q^2)\setminus\cH_q$ the two points of $\ell$ fixed by $\gamma$ (which are the center of $\gamma$ and the intersection between $\ell$ and the axis of $\gamma$; see Lemmma \ref{classificazione}).
Since $d>2$ and $\delta$ acts with orbits of length $d$ on $\ell\setminus\{Q,R\}$, $\delta$ cannot commute with $\gamma$ unless $\{Q,R\}=\{\tilde{Q},\tilde{R}\}$, which implies $\delta$ being of type (B1) and hence $d\mid(q+1)$.
If $\delta$ and $\gamma$ do not commute, then $G$ induces a dihedral group $\bar{G}\leq\PGL(2,q)$ of order $\frac{d}{\gcd(2,d)}\cdot2$, and $G$ is dihedral itself.

Conversely, we show that abelian and dihedral groups $G$ as in the statement actually exist in ${\rm SU}^{\pm}(2,q)\setminus H$.
To this aim, we make use of other models of $\cH_q$ and provide groups $G$ which fix $P$ whose order is coprime to $\frac{q+1}{2}$; this assures that $G\leq{\rm SU}^{\pm}(2,q)$.

A cyclic group $G$ of order $2d$ with $d\mid(q-1)$ and $d\nmid\frac{q-1}{2}$ is provided as follows: $\cH_q$ has equation \eqref{M2}, $P=(0,1,0)$, and $G$ is generated by $\diag(a^{q+1},a,1)$, where $0(a)=2d$; if $G\leq H$, multiply its generator with $\diag(-1,1,1)\in{\rm SU}^{\pm}(2,q)\setminus H$.

A dihedral group $G$ of order $2d$ with $d\mid(q-1)$ is provided as follows: $\cH_q$ has equation \eqref{M2}, $P=(0,1,0)$, and $G$ is generated by $\diag(a^{q+1},a,1)$ and
$$\begin{pmatrix} 0 & 0 & 1 \\ 0 & 1 & 0 \\ 1 & 0 & 0 \end{pmatrix},$$
where $o(a)=d$.

A cyclic group $G$ of order $2d$ with $2<d\mid(q+1)$ is provided as follows: $\cH_q$ has equation \eqref{M1}, $P=(0,0,1)$, and $G$ is generated by $\diag(\lambda,\lambda^i,1)$, where $o(\lambda)=2d$ and $\gcd(2d,i)=1$. After $\lambda$ is chosen, there exists only one value $\bar{i}\in\{2,\ldots,2d-1\}$ such that $\diag(\lambda,\lambda^{\bar{i}},1)\in H$; hence we can choose $i$ such that $G\not\leq H$.

An abelian group $G=C_d\times C_2$ with $d\mid(q+1)$ is provided as follows: $\cH_q$ has equation \eqref{M1}, $P=(0,0,1)$, $G_H$ is generated by $\diag(\lambda,\lambda^i,1)$ as in the previous point, and $G$ is generated by $G_H$ together with $\diag(-1,1,1)\in{\rm SU}^{\pm}(2,q)\setminus H$.

A dihedral group of order $2d$ with $d\mid(q+1)$ is provides as follows: $\cH_q$ has equation \eqref{M1}, $P=(0,0,1)$, $G_H$ is generated by $\diag(\lambda,\lambda^i,1)$ as in the previous point, and $G$ is generated by $G_H$ together with
$$ \begin{pmatrix} 0 & 1 & 0 \\ 1 & 0 & 0 \\ 0 & 0 & 1 \end{pmatrix}. $$

\item Suppose $G_H=Dic_m=\langle\delta,\epsilon\rangle$ with $m\mid\frac{q-1}{2}$, $o(\delta)=2m$, and $o(\epsilon)=4$. Note that both $\delta$ and $\epsilon$ are of type (B2); let $Q$ and $R$ be the points on $\ell\cap\cH_q$ fixed by $\delta$. Let $\alpha\in G\setminus G_H$. Since $\langle\delta\rangle$ is normal in $G$, $\alpha$ acts on $\{Q,R\}$. Up to replacing $\alpha$ with $\alpha \epsilon$, we can assume that $\alpha(Q)=Q$ and $\alpha(R)=R$.
Since the pointwise stabilizer $S$ of $\{Q,R\}$ in $\PGU(3,q)$ is cyclic, $\langle\delta,\alpha\rangle$ is cyclic of order $4m$; up to replacing $\alpha$ with a generator of $\langle\delta,\alpha\rangle$, we can assume that $\alpha$ is of type (B2) and has order $4m$.
Therefore $G=\langle\alpha,\epsilon\rangle$ with $\alpha^{2m}=\epsilon^2$.
Since $G\not\leq H$ and $|S\cap H|=q-1$, we have $o(\alpha)=4m\nmid(q-1)$.

Such a subgroup $G$ actually exists in ${\rm SU}^{\pm}(2,q)$ and is determined uniquely up to conjugation, as follows.
Let $\cH_q$ have equation \eqref{M2}; up to conjugation we have $P=(0,1,0)$, $Q=(1,0,0)$, $R=(0,0,1)$. The only element of order $4m$ in $\PGU(3,q)$ fixing $\{P,Q,R\}$ pointwise is $\alpha=\diag(a^{q+1},a,1)$, where $a$ is a primitive $4m$-th root of unity.
Any element of order $4$ in $\PGU(3,q)$ fixing $P$ and interchanging $Q$ and $R$ has the form
$$ \epsilon=\begin{pmatrix} 0 & 0 & \gamma \\ 0 & 1 & 0 \\ -\gamma^{-1} & 0 & 0 \end{pmatrix}, $$
where $\gamma^{q+1}=1$.
By direct checking, $\epsilon^{-1}\alpha\epsilon=\alpha^{2m-1}$ and $G=\langle\alpha,\epsilon\rangle$ has order $8m$.
Since $G\leq\cM_q$ and $\gcd(|G|,\frac{q+1}{2})=1$, we have $G\leq{\rm SU}^{\pm}(2,q)$. Also, the assumptions $m\mid\frac{q-1}{2}$ and $m\nmid\frac{q-1}{4}$ imply $\alpha\notin H$ and $\alpha^2\in H$, so that $G_H=\langle\alpha^2,\epsilon\rangle\cong Dic_m$.
The elements of $G$ are $\alpha^i$ and $\alpha^i\epsilon$, with $i=0,\ldots,4m-1$. By direct checking, $\alpha^{2m}$ and $\alpha^i \epsilon$ with odd $i$ are involutions and hence of type (A), while $\alpha^j$ with $j\ne0,2m$ and $\alpha^k \epsilon$ with even $k$ are of type (B2).

\item Suppose $G_H=Dic_m=\langle\delta,\epsilon\rangle$ with $m\mid\frac{q+1}{2}$, $o(\delta)=2m$, and $o(\epsilon)=4$.
Denote by $Q$ and $R$ the points other than $P$ fixed by $\delta$; we have $Q,R\in\ell\setminus\cH_q$.
Let $\alpha\in G\setminus G_H$. If $o(\alpha)=4m$, then $o(\alpha)\mid(q^2-1)$ and $o(\alpha)\nmid(q+1)$, so that $\alpha$ is of type (B2) and $\alpha^2$ is of type (A), a contradiction to the fact that $\iota$ is the only element of type (A) in $H$; hence, $o(\alpha)\ne4m$.
Since $\langle\delta\rangle$ is normal in $G$, the subgroup $K=\langle\delta,\alpha\rangle$ has order $4m$. As $K$ is not cyclic, we have shown above that either $K$ is a direct product $C_{2m}\times C_2$, or $K$ is a dihedral group $C_{2m}\rtimes C_2$.
We can then assume that $\alpha$ is an involution.

We show that we can also assume $K=\langle\delta\rangle\times\langle\alpha\rangle\cong C_{2m}\times C_2$.
The number of subgroups of order $4$ generated elements of $G_H\setminus \langle\delta\rangle$ is equal to $m$ and hence is odd. This implies that $\alpha$ normalizes $\langle\zeta\rangle$ for some $\zeta\in G_H\setminus\langle\delta\rangle$ with $o(\zeta)=4$; up to conjugation, $\zeta=\epsilon$.
Let $\cH_q$ have equation \eqref{M1} and assume up to conjugation that $P=(0,0,1)$, $Q=(1,0,0)$, and $R=(0,1,0)$.
If $\alpha(Q)=Q$ and $\alpha(R)=R$, then $\alpha$ is represented by a diagonal matrix and commutes with $\delta$. Suppose that $\alpha$ and $\delta$ do not commute, so that $\alpha$ interchanges $Q$ and $R$. Then
\begin{equation}\label{costruzdiciclico}
\delta=\begin{pmatrix} \lambda & 0 & 0 \\ 0 & \lambda^{-1} & 0 \\ 0 & 0 & 1 \end{pmatrix},\quad\alpha=\begin{pmatrix} 0 & a & 0 \\ a^{-1} & 0 & 0 \\ 0 & 0 & 1 \end{pmatrix},\quad \epsilon=\begin{pmatrix} 0 & e & 0 \\ -e^{-1} & 0 & 0 \\ 0 & 0 & 1 \end{pmatrix},
\end{equation}
for some $(q+1)$-th roots of unity $\lambda,a,e$ with $o(\lambda)=2m$. Since $\alpha$ normalizes $\epsilon$, either $a^2=-e^2$ if $\alpha\epsilon\alpha=\epsilon$, or $a^2=e^2$ if $\alpha\epsilon\alpha=\epsilon^{-1}$. In the former case, $\alpha\epsilon$ is a diagonal matrix, so that $\alpha\epsilon$ fixes $\{Q,R\}$ pointwise; but $o(\alpha\epsilon)=4$, so that $\alpha\epsilon$ is of type (B2): a contradiction. Hence, $\alpha\epsilon\alpha=\epsilon^{-1}$; this implies $o(\alpha\epsilon)=2$, and either $\alpha\epsilon=\diag(1,-1,1)$ or $\alpha\epsilon=\diag(-1,1,1)$. We can then replace $\alpha$ with $\alpha\epsilon$, so that $K=\langle\delta\rangle\times\langle\alpha\rangle$
and $G$ is a product $K\langle\epsilon\rangle\cong (C_{2m}\times C_2) C_4$, with $|K\cap\langle\epsilon\rangle|=2$.

Such a group $G=(\langle\delta\rangle\times\langle\alpha\rangle)\langle\epsilon\rangle\cong (C_{2m}\times C_2) C_4$ actually exists in ${\rm SU}^{\pm}(2,q)$, as the choice in \eqref{costruzdiciclico} shows. The elements in $\langle\delta\rangle\times\langle\alpha\rangle$ of order greater than $2$ are of type (B1); the three involutions in $\langle\delta\rangle\times\langle\alpha\rangle$ are of type (A); the elements $\delta^i \epsilon$ of $G_H\setminus\langle\delta\rangle$ have order $4$ and are of type (B2); the remaining elements have the form $\delta^i \alpha\epsilon$, are involutions, and are of type (A).

\item Suppose $G_H=\SL(2,3)$ with $p\geq5$.
Then $\bar G$ is a subgroup of order $24$ with a subgroup $\bar{G}_H\cong\mathbf{A}_4$; from Theorem \ref{Di}, $\bar G\cong\mathbf{S}_4$.
By direct checking with MAGMA \cite{MAGMA}, the unique groups $L$ admitting a normal subgroup isomorphic to $\SL(2,3)$ and such that the factor group of $L$ over the unique involution of $\SL(2,3)$ is isomorphic to $\mathbf{S}_4$ are $SmallGroup(48,28)\cong\hat{\Sigma}_4$ and $SmallGroup(48,29)$.

We show that, if $L_1$ and $L_2$ are subgroups of ${\rm SU}^{\pm}(2,q)$ containing $G_H$ with index $[L_1:G_H]=[L_2:G_H]=2$, then $L_1=L_2$.
By direct inspection on $SmallGroup(48,28)$ and $SmallGroup(48,29)$, both $L_1$ and $L_2$ are generated by $G_H$ together with any element of order $8$, whose square lies in $G_H$. Also, any cyclic subgroup $C_4$ of order $4$ of $G_H$ is contained is contained both in a cyclic subgroup $C_8^1$ of order $8$ of $L_1$ and in a cyclic subgroup $C_8^2$ of order $8$ of $L_2$.
The group $C_4$ is generated by an element of type (B2) with two fixed points $Q,R\in\ell\cap\cH_q$; thus, $C_8^1$ and $C_8^2$ act on $\{Q,R\}$. If a generator $\alpha_i$ of $C_8^i$ ($i\in\{1,2\}$) interchanges $Q$ and $R$, then $\alpha_i^2$ is of type (A), a contradiction to Lemma \ref{classificazione}. Thus, both $C_8^1$ and $C_8^2$ fix $\{Q,R\}$ pointwise. Since the pointwise stabilizer of $\{Q,R\}$ in $\PGU(3,q)$ is cyclic, this implies $C_8^1=C_8^2$ and hence $L_1=L_2$.

If $8\mid(q-1)$, then by Theorem \ref{sl2q} there $H$ already contains a subgroup $\hat{\Sigma}_4 \cong SmallGroup(48,28)$ having a subgroup isomorphic to $\SL(2,3)$. Hence, there exists no subgroup $G$ of ${\rm SU}^{\pm}(2,q)$ with $G_H\cong\SL(2,3)$.

If $8\nmid(q-1)$, then $G\leq{\rm SU}^{\pm}(2,q)$ with $G_H\cong\SL(2,3)$ does exist, and can be constructed as follows.
Let $\cH_q$ be given by Equation \eqref{M2}; up to conjugation, $P=(0,1,0)$ and $\ell: Y=0$.
Choose $\lambda,\mu\in\mathbb{F}_q$ and $c,e\in\mathbb{F}_{q^2}$ such that $\lambda^2=-1$, $\mu^2=\lambda$, $c^2=\frac{\lambda+1}{2}$, and $e=\mu c$.
Define
$$ \alpha_1=\begin{pmatrix} -1 & 0 & 0 \\ 0 & \lambda & 0 \\ 0 & 0 & 1 \end{pmatrix},\quad \alpha_2=\begin{pmatrix} 0 & 0 & -\lambda c \\ 0 & 1 & 0 \\ -\frac{\lambda}{c} & 0 & 0 \end{pmatrix},\quad \alpha_3=\begin{pmatrix} 0 & 0 & c \\ 0 & 1 & 0 \\ -\frac{1}{c} & 0 & 0 \end{pmatrix}, $$
$$ 
\xi=\begin{pmatrix} -\frac{\lambda+1}{2} & 0 & \frac{\lambda-1}{2}c \\ 0 & 1 & 0 \\ c & 0 & \frac{1-\lambda}{2} \end{pmatrix},\quad
\gamma=\begin{pmatrix} 0 & 0 & e \\ 0 & 1 & 0 \\ e^{-1} & 0 & 0 \end{pmatrix}. $$
By direct checking, the following holds.
\begin{itemize}
\item $\alpha_1,\alpha_2,\alpha_3$ have order $4$; $\alpha_1^2=\alpha_2^2=\alpha_3^3=\iota=\diag(1,-1,1)$; $\alpha_1 \alpha_2 \alpha_3 = id$. Hence, $Q_8=\langle\alpha_1,\alpha_2,\alpha_3\rangle$ is a quaternion group.
\item $\alpha_1,\alpha_2,\alpha_3\in\cM_q$ since they preserve the equation \eqref{M2} of $\cH_q$ and fix $P$; $\alpha_1,\alpha_2,\alpha_3$ are of type (B2). The fixed points other than $P$ are $P_1=(1,0,0)$ and $P_3=(0,0,1)$ for $\alpha_1$; $Q_1=(-c,0,1)$ and $Q_3=(c,0,1)$ for $\alpha_2$; $R_1=(-\lambda c,0,1)$ and $R_3=(\lambda c,0,1)$ for $\alpha_3$.
\item $\xi$ has order $3$ and $\xi\in\cM_q$. For $i=\{1,3\}$, we have $\xi(P_i)=Q_i$, $\xi(Q_i)=R_i$, and $\xi(R_i)=R_i$; since the pointwise stabilizer of two points in $\ell\cap\cH_q$ is cyclic, this implies that $\xi$ normalizes $Q_8$.
Also, $\xi$ does not commute with $\alpha_j$.
Then $\langle\alpha_1,\alpha_2,\alpha_3,\xi\rangle=Q_8\rtimes C_3$ is isomorphic to $\SL(2,3)$.
\item $\xi$ has no fixed points in $\PG(2,\bar{\mathbb{F}}_{q^2})\setminus\ell$; this implies by Lemma \ref{classificazione} that $\xi$ is of type (B1).
\item $\gamma$ is an involution of $\cM_q$ satisfying $\gamma(P_1)=P_3$, $\gamma(Q_1)=R_1$, and $\gamma(Q_3)=R_3$; this implies that $\gamma$ normalizes $Q_8$.
Also, $\gamma\xi\gamma=\alpha_2\xi^2\in Q_8\rtimes C_3$; this implies that $\gamma$ normalizes $Q_8\rtimes C_3\cong\SL(2,3)$.
\end{itemize}
Therefore, $G=\langle\alpha_1,\alpha_2,\alpha_3,\xi,\gamma\rangle$ is a subgroup of $\cM_q$ of order $48$, with a subgroup of index $2$ isomorphic to $\SL(2,3)$. Since $G$ contains a Klein four-group $\langle\iota,\gamma\rangle$, we have $G\cong SmallGroup(48,29)$.
We have $\gcd(|G|,\frac{q+1}{2})\in\{1,3\}$; recall that $\cM_q\times C_{\frac{q+1}{2}}$. If $\gcd(|G|,\frac{q+1}{2})=1$, then $G\leq{\rm SU}^{\pm}(2,q)$. If $\gcd(|G|,\frac{q+1}{2})=3$, then all elements of order $3$ in $G$ are of type (B1) as they are conjugated to $\xi$; this implies again $G\leq{\rm SU}^{\pm}(2,q)$, because any element $\psi$ of order $3$ in $\cM_q\setminus{\rm SU}^{\pm}(2,q)$ is of type (A).

In fact, let $\cH_q$ have equation \eqref{M1}; up to conjugation, $P=(0,0,1)$ and $\psi$ fixes $Q=(1,0,0)$ and $R=(0,1,0)$. This implies $\psi=\diag(u,v,1)$ with $u^3=v^3=1$. If $\psi$ is not of type (A), then $u\ne1$, $v\ne1$, $u\ne v$; hence, $v=u^{-1}$ and $\psi\in{\rm SU}^{\pm}(2,q)$.
\end{itemize}
\end{proof}


We will make use of the following remark, which can be easily proved in analogy to \cite[Remark 4.1]{DVMZ}.

\begin{remark}\label{rimarco}
Let $G\leq\cM_q$ be such that $G/G_{\Omega}$ is generated by elements whose order is coprime to $|G_{\Omega}|$. Then $G=G_{\pm}\times G_{\Omega}$.
If in addition $G/G_\Omega$ is generated by elements of odd order, then $G=G_H\times G_\Omega$.
\end{remark}

Now we compute the genera of quotient curves $\cH_q/G$ for all subgroups $G$ of $\cM_q={\rm SU}^{\pm}(2,q)\times\Omega$.
The factor group $G/G_{\Omega}$ is isomorphic to a subgroup of ${\rm SU}^{\pm}(2,q)$. Hence, we will consider the different possibilities for $G/G_{\Omega}$ given by Proposition \ref{sottogruppiSU} and Theorem \ref{sl2q}.

\begin{lemma}\label{gianoti}
Let $G\leq\cM_q$ be such that one of the following cases holds:
\begin{enumerate}
\item $G/G_{\Omega}$ is cyclic of order a divisor of $q+1$ different from $2$.
\item $G/G_{\Omega}$ is dicyclic of order $4m$ with $1<m\mid\frac{q+1}{2}$.
\item $G/G_{\Omega}\cong C_d\times C_2$ where $d\mid(q+1)$ and $d$ is even.
\item $G/G_{\Omega}$ is dihedral of order $2d$ with $1<d\mid(q+1)$.
\item $G/G_{\Omega}\cong Dic_m\rtimes C_2$ with $1<m\mid\frac{q+1}{2}$.
\end{enumerate}
Then $G$ is contained in the maximal subgroup of $\PGU(3,q)$ stabilizing a self-polar triangle $T\subset\PG(2,q^2)\setminus\cH_q$.

If $G\leq \cM_q$ is such that $G/G_{\Omega}\cong E_{p^k}\rtimes C$, where $E_{p^k}$ is elementary abelian of order $p^k$ and $C$ is cyclic, then $G$ is contained in the maximal subgroup of $\PGU(3,q)$ stabilizing a point of $\cH_q(\mathbb{F}_{q^2})$.
\end{lemma}

\begin{proof}
If $\langle\alpha_1 G_{\Omega},\ldots,\alpha_r G_{\Omega}\rangle$ is a normal subgroup of $G/G_\Omega$, then $\langle\alpha_1,\ldots,\alpha_r\rangle$ is a normal subgroup of $G$, because $G_{\Omega}$ is central in $G$.
\begin{itemize}
\item Let $G/G_\Omega\cong E_{p^k}\rtimes C$. Then $E_{p^k}$ has a unique fixed point on $\cH_q(\fqs)$; see \cite[Lemma 11.129]{HKT}. As $E_{p^k}$ is normal in $G$, $G$ fixes this point.

\item Let $G/G_{\Omega}=\langle\alpha G_\Omega\rangle\times\langle\gamma G_\Omega\rangle$ satisfy assumption (3), with $\alpha^d,\gamma^2\in G_\Omega$. We show that $o(\alpha),o(\gamma)\mid(q+1)$; this implies that $\alpha$ and $\gamma$ are of type (A) or (B1), $\langle\alpha,\gamma\rangle$ fixes pointwise a triangle $T=\{P,Q,R\}$ with $Q,R\in\ell(\mathbb{F}_{q^2})\setminus\cH_q$, and hence $G$ fixes $T$ pointwise, the claim.

Suppose by contradiction that $o(\alpha)\nmid(q+1)$. By Lemma \ref{classificazione}, $o(\alpha)\mid 2(q+1)$, $\alpha$ is of type (B2), and $\alpha^2$ is of type (A) with center $P$. Let $Q,R\in\ell\cap\cH_q$ be the fixed points of $\alpha$ other than $P$. Since the pointwise stabilizer of $\{Q,R\}$ is cyclic unlikely $G/G_{\Omega}$, $\gamma$ interchanges $Q$ and $R$. Let $\cH_q$ have equation \eqref{M2}; up to conjugation, $P=(0,1,0)$, $Q=(1,0,0)$, and $R=(0,0,1)$; hence,
$$ \alpha=\begin{pmatrix} -1 & 0 & 0 \\ 0 & a & 0 \\ 0 & 0 & 1 \end{pmatrix},\quad
\gamma=\begin{pmatrix} 0 & 0 & c \\ 0 & 1 & 0 \\ d & 0 & 0 \end{pmatrix},
 $$
where $c^{q+1}=d^{q+1}=a^{2(q+1)}=1\ne a^{q+1}$. Then $\alpha\gamma\ne\gamma\alpha$, a contradiction.
Suppose that $o(\gamma)\nmid(q+1)$. Then swap the roles of $\alpha$ and $\gamma$ in the argument above to obtain a contradiction.

\item Let $G/G_\Omega=\langle\alpha G_\Omega\rangle$ satisfy assumption (1), with $|G/G_\Omega|=d$. By Lemma \ref{classificazione} $o(\alpha)\mid(q^2-1)$. Together with $o(\alpha)\mid d(q+1)$, this yields $o(\alpha)\mid 2(q+1)$. If $o(\alpha)\nmid(q+1)$, then $\alpha$ is of type (B2) and $\alpha^2\in G_\Omega$, a contradiction.
Hence, $o(\alpha)\mid(q+1)$, and $\alpha$ is of type (A) or (B1).
Since $\langle\alpha\rangle$ is normal in $G$, $G$ acts on the points fixed by $\alpha$; as $G$ fixes $P$, this implies that $G$ stabilizes a self-polar triangle $\{P,Q,R\}\subset\PG(2,q^2)\setminus\cH_q$.

\item Let $G/G_\Omega=\langle\alpha G_\Omega,\epsilon G_\Omega\rangle\cong Dic_n$ or $G/G_\Omega=\langle\alpha G_\Omega,\epsilon G_\Omega,\xi G_\Omega\rangle\cong Dic_n\rtimes C_2$ satisfy assumption (2) or (5), respectively, with $6\leq o(\alpha G_\Omega)=2n\mid(q+1)$.
Using the normality of $\langle\alpha\rangle$ in $G$ and arguing as in the previous point, we have that $\alpha$ is of type (A) or (B1), and $G$ stabilizes a self-polar triangle $T\subset\PG(2,q^2)\setminus\cH_q$.

\item Let $G/G_\Omega=\langle\alpha\rangle\rtimes\langle\gamma\rangle$ satisfy assumption (4), with $\alpha^d,\gamma^2\in G_\Omega$.
If $d=2$, then $G$ satisfies also assumption (3) and the claim was already proved.
If $d>2$, then $\langle\alpha\rangle$ is normal in $G$; arguing as in the previous point, $\alpha$ is of type (A) or (B1), and $G$ stabilizes a self-polar triangle $T\subset\PG(2,q^2)\setminus\cH_q)$.
\end{itemize}
\end{proof}

Lemma \ref{gianoti} provides cases for $G/G_\Omega$ which do not need to be considered in the following, since $G$ is contained in a maximal subgroup of $\PGU(3,q)$ for which $g(\cH_q/G)$ has already been computed in the literature.
Namely, if $G/G_\Omega\cong E_{p^k}\rtimes C$, then $g(\cH_q/G)$ is computed in \cite[Theorem 1.1]{BMXY};
if $G/G_\Omega$ satisfies one of the assumptions (1) to (5), then $g(\cH_q/G)$ is computed in \cite[Proposition 3.4]{DVMZ}.

\begin{proposition}
Let $G\leq\cM_q$ be such that $G/G_\Omega\cong\SL(2,5)$, with $q^2\equiv1\pmod5$. Then
$$ g(\cH_q/G)= \frac{(q+1)(q-1-2\omega)+180\omega-20r-48s}{240\omega},$$
where
$$r=\begin{cases} 4\omega & \textrm{if}\quad3\mid(q-1), \\
q+1+2\omega & \textrm{if}\quad3\mid q, \\
0 & \textrm{if}\quad3\mid(q+1),\; 3\nmid\omega,\\
2(q+1) & \textrm{if}\quad3\mid\omega; \end{cases}\quad
s=\begin{cases} 2\omega & \textrm{if}\quad5\mid(q-1), \\
0 & \textrm{if}\quad5\mid(q+1),\; 5\nmid\omega, \\
q+1 & \textrm{if}\quad5\mid\omega. \end{cases} $$
\end{proposition}

\begin{proof}
By Remark \ref{rimarco}, $G=G_H \times G_\Omega$ with $G_H\cong\SL(2,5)$.
The nontrivial elements of $G$ are as follows.
\begin{itemize}
\item $20$ elements of order $3$ in $G_H$; they are of type (B2), (C), or (B1), according to $3\mid(q-1)$, $3\mid q$, or $3\mid(q+1)$, respectively.
\item $2\omega-1$ elements in $\langle\iota\rangle\times G_\Omega$; they are of type (A).
\item $30\omega$ elements obtained as the product of an element of order $4$ in $G_H$ by an element of $G_\Omega$; they are of type (B2).
\item $20\omega$ elements obtained as the product of an element of order $6$ in $G_H$ by an element of $G_\Omega$; they are of type (B2), (E), or (B1), according to $3\mid(q-1)$, $3\mid q$, or $3\mid(q+1)$, respectively.
\item $48\omega$ elements obtained as the product of an element $\eta\in G_H$ of order $5$ or $10$ by an element $\theta\in G_\Omega$. If $5\mid(q-1)$, they are of type (B2). If $5\mid(q+1)$ and $5\nmid\omega$, they are of type (B1).

If $5\mid\omega$, $48$ of them are of type (A), the other ones are of type (B1). Namely, if $o(\eta)=10$ then $\eta\theta$ is of type (B1); if $o(\eta)=5$ and $\{P,Q,R\}$ are the fixed points of $\eta$, then there are exactly $2$ choices for $\theta\in G_\Omega$ such that $\eta\theta$ is of type (A), one with center $Q$, the other with center $R$.
\item $20(\omega-1)$ elements obtained as the product of an element of order $3$ in $G_H$ by a nontrivial element of $G_\Omega$; they are of type (B2) or (E) if $3\mid(q-1)$ or $3\mid q$, respectively.
If $3\mid(q+1)$, either they are all of type (B1) or there are $48$ of them of type (A), according to $5\nmid\omega$ or $5\mid\omega$, arguing as in the previous case.
\end{itemize}
The claim follows by direct computation using the Riemann-Hurwitz formula and Theorem \ref{caratteri}.
\end{proof}

\begin{proposition}
Let $G\leq\cM_q$ be such that $G/G_\Omega\cong\hat{\Sigma}_4 \cong SmallGroup(48,28)$, with $p\geq5$ and $8\mid(q-1)$.
Then
$$g(\cH_q/G)=\frac{(q+1)(q-1-2\omega)+36\omega-16r}{96\omega},\quad\textrm{where}\quad
r=\begin{cases} 2\omega & \textrm{if}\quad3\mid(q-1), \\
0 & \textrm{if}\quad3\mid(q+1),\;3\nmid\omega, \\
q+1 & \textrm{if}\quad3\mid\omega. \end{cases}$$
\end{proposition}

\begin{proof}
By Remark \ref{rimarco} and Proposition \ref{sottogruppiSU}, $G=G_H\times G_\Omega$ with $G_H\cong SmallGroup(48,28)$.
The nontrivial elements of $G$ are as follows. Since $SU^{\pm}(2,q)$ has no elements of type (A) with odd order, the type of any element in $G_H$ is uniquely determined by its order. The product of $\iota$ by an element of $G_\Omega$ has type (A), and the product of an element of type (B2) in $G_H$ by an element of $G_\Omega$ has type (B2).
The product of an element $\eta$ of type (B1) in $G_H$ by an element of $G_\Omega$ has type (B1), unless $o(\eta)\mid\omega$; if $o(\eta)\mid\omega$, then $G_\Omega$ contains exactly $2$ elements $\theta_1,\theta_2$ such that $\eta\theta_1$ and $\eta\theta_2$ are of type (A); here, this argument applies to the elements of order $3$.
Now the claim follows by direct computation with the Riemann-Hurwitz formula and Theorem \ref{caratteri}.
\end{proof}

\begin{proposition}
Let $G\leq\cM_q$ be such that $G/G_\Omega\cong\SL(2,3)$ with $p\geq5$. If $3\mid(q-1)$, then
\begin{equation}\label{sl23caso1}
g(\cH_q/G)=\frac{(q+1)(q-1-2\omega)+4\omega}{48\omega}.
\end{equation}
If $3\mid(q+1)$, then one of the following cases holds:
\begin{equation}\label{sl23caso2}
g(\cH_q/G)=\frac{(q+1)(q-1-2\omega)+36\omega-8(q+1)(\gcd(3,\omega)-1)}{48\omega};
\end{equation}
\begin{equation}\label{sl23caso3}
g(\cH_q/G)=\frac{(q+1)(q-9-2\omega)+36\omega}{48\omega},\quad \textrm{with}\quad 3\nmid\omega;
\end{equation}
\begin{equation}\label{sl23caso4}
g(\cH_q/G)=\frac{(q+1)(q-1-2\omega)+36\omega}{48\omega},\quad \textrm{with}\quad 3\mid\omega,\quad 3\mid\frac{q+1}{\omega}.
\end{equation}
All cases {\rm \eqref{sl23caso1}} to {\rm \eqref{sl23caso4}} actually occur, for some $G$ as in the assumptions.
\end{proposition}

\begin{proof}
Assume $3\mid(q-1)$. By Remark \ref{rimarco}, $G=G_H\times G_\Omega$ with $G_H\cong\SL(2,3)$. The nontrivial elements of $G$ are as follows: $2\omega-1$ elements of type (A) in $\langle\iota\rangle\times G_\Omega$; $22\omega$ elements of order divisible by $3$ or $4$, which are of type (B2).
Equation \eqref{sl23caso1} follows by the Riemann-Hurwitz formula and Theorem \ref{caratteri}.

For the rest of the proof, assume $3\mid(q+1)$. Let $G/G_\Omega=Q_8\rtimes \langle\xi G_\Omega\rangle$, where $Q_8$ is a quaternion group and $\xi\notin G_\Omega$ satisfies $\xi^3\in G_\Omega$. Since $Q_8$ and $G_\Omega$ have coprime orders, $Q_8$ is induced by a subgroup $\langle\alpha_1,\alpha_2\rangle$ of $G$ isomorphic to $Q_8$.

Suppose that there exists $\xi\in G_H$ inducing $\xi G_\Omega$; this can be assumed when $3\nmid\omega$.
Then $G=G_H\times G_\Omega$, and the nontrivial elements of $G$ are as follows: $2\omega-1$ elements of type (A) in $\langle\iota\rangle\times G_\Omega$; $6\omega$ elements of order divisible by $4$, which are of type (B2). If $3\nmid\omega$, any other element is of type (B1). If $3\mid\omega$, then there are $8\cdot2$ elements of order $3$ and type (A): namely, for any element $\eta\in G_H$ of order $3$ there are exactly $2$ elements $\theta_1,\theta_2\in G_\Omega$ such that $\eta\theta_1,\eta\theta_2$ are of type (A); any other element is of type (B1).
Equation \eqref{sl23caso2} follows by the Riemann-Hurwitz formula and Theorem \ref{caratteri}.

For the rest of the proof, we can assume that $\xi\notin G_H$ for any $\xi\in G$ inducing $\xi G_\Omega$. As the subgroups of $\SL(2,3)$ of order $3$ are conjugated by elements of $Q_8$, no element of $G_H$ induces an element of order $3$ in $G/G_\Omega$.
\begin{itemize}
\item
Suppose that $o(\xi)=3$. Since $\xi\notin G_H$ and $\xi\notin G_\Omega$, this implies that $\xi$ is of type (A) with center on $\ell$ and axis passing through $P$. Note that $G_H=Q_8$.
Note also that $3\nmid\omega$; otherwise, $G_\Omega$ has an element $\rho$ of order $3$, and either $\xi\rho$ or $\xi\rho^2$ is an element of type (B1) and order $3$ lying in $G_H$, a contradiction.
The nontrivial elements of $G$ are as follows: $2\omega-1$ elements of type (A) in $\langle\iota\rangle\times G_\Omega$; $8$ elements of order $3$ and type (A); $6\omega$ elements of order a multiple of $4$ and type (B2); any other element is of type (B1).
Equation \eqref{sl23caso3} follows by the Riemann-Hurwitz formula and Theorem \ref{caratteri}.

Such a group $G$ actually exists in $\cM_q$. In fact, let $3\nmid\omega$ and $(Q_8\rtimes C_3)\times C_{\omega}$ be the subgroup isomorphic to $\SL(2,3)\times G_\Omega$ constructed above, with $Q_8\rtimes C_3\leq H$. Let $\eta$ be a generator of $C_3\leq H$ and $\rho$ be an element of order $3$ in $\Omega$. Then $\eta\rho$ is an element of order $3$ and type (A) not in $\Omega$, such that $G:=(Q_8\rtimes \langle\eta\rho\rangle)\times C_{\omega}$ is the desired group.
\item
Now suppose that $o(\xi)>3$. Up to composing with an element of $G_\Omega$, we can assume that $\xi$ is a $3$-element, of order $3^k$, $k\geq2$. As $\xi^3$ is a nontrivial element of $G_\Omega$, $\xi$ is not of type (A), and hence is of type (B1).
As $G_{\Omega}$ is cyclic, $\langle\xi^3\rangle$ is the Sylow $3$-subgroup of $G_{\Omega}$.
Then $G\cong(Q_8\rtimes C_{3^k})\times C_{\omega/3^{k-1}}$, where $C_{3^k}=\langle\xi\rangle$ and $3^k\nmid\omega$.
The nontrivial elements of $G$ are as follows.
\begin{itemize}
\item Elements of $Q_8\rtimes C_{3^k}$. There are $2\cdot 3^{k-1}-1$ elements of type (A) in $\langle\iota\rangle\times C_{3^{k-1}}$; $6\cdot 3^{k-1}$ elements of type (B2), as the product of an element of order $4$ in $Q_8$ by an element of $C_{3^{k-1}}$; $2(3^k-3^{k-1})$ elements of type (B1) in $(\langle\iota\rangle\times C_{3^k})\setminus (\langle\iota\rangle\times C_{3^{k-1}})$.

The $6(3^k-3^{k-1})$ elements $\sigma$ obtained as the product of an element $\alpha$ of order $4$ in $Q_8$ by an element $\gamma$ of order $3^k$ in $C_{3^k}$ are of type (B1).
In fact, $o(\sigma)\in\{3^k,2\cdot 3^k,4\cdot 3^k\}$. If $o(\sigma)=4\cdot 3^k$, then $\sigma$ is of type (B2) and $o(\sigma^4)\mid(q+1)$, so that $\sigma^4\in G_H$ with $o(\sigma)^4=3^k$, a contradiction.
If $o(\sigma)=2\cdot 3^k\mid(q+1)$, then $\sigma$ cannot be of type (A) with center on $\ell$, because $\sigma^{3^k}=\iota$ has center $P$; hence, $\sigma$ is of type (B1).
If $o(\sigma)=3^k$, then by Schur-Zassenhaus Theorem $\langle\sigma\rangle$ is conjugated to $C_{3^k}$; hence, $\sigma$ is of type (B1).

\item Elements of $G\setminus(Q_8\rtimes C_{3^k})$.
There are $2\cdot 3^{k-1}\cdot \left(\frac{\omega}{3^{k-1}}-1\right)$ elements of type (A) in $\langle\iota\rangle\times C_{3^{k-1}}\times C_{\omega/3^{k-1}}$; $6\cdot 3^{k-1}\cdot \left(\frac{\omega}{3^{k-1}}-1\right)$ elements of type (B2) obtained as the product of an element of type (B2) in $Q_8\rtimes C_{3^k}$ by a non trivial element of $C_{\omega/3^{k-1}}$; $8(3^k-3^{k-1})\left(\frac{\omega}{3^{k-1}}-1\right)$ elements of type (B1) obtained as the product of an element either trivial or of type (B1) in $Q_8\rtimes C_{3^k}$ by a nontrivial element of $C_{\omega/3^{k-1}}$.
\end{itemize}
Equation \eqref{sl23caso4} follows by the Riemann-Hurwitz formula and Theorem \ref{caratteri}.

Such a group actually exists in $\cM_q$.
In fact, consider a group $(Q_8\rtimes C_3)\times C_{\omega/3^{k-1}}$ with $Q_8\rtimes C_3\leq H$ and $C_{\omega/3^{k-1}}\leq\Omega$, as described above; let $\xi$ be a generator of $C_3$, and $\sigma$ be an element of $\Omega$ of order $3^k$. Then replace $\xi$ with $\xi\sigma$, that is, define $G=(Q_8\rtimes\langle\xi\sigma\rangle)\times C_{\omega/3^{k-1}}$.
\end{itemize}
\end{proof}

\begin{proposition}\label{ciclicoq-1}
Let $G\leq\cM_q$ be such that $G/G_\Omega$ is cyclic of order $d$ with $d\mid(q-1)$. Then either $d=2$ and $G$ stabilizes pointwise a self-polar triangle $\{P,Q,R\}\subset\PG(2,q^2)\setminus\cH_q$, or
$$ g(\cH_q/G)=\frac{(q-1)(q+1-\omega\cdot\gcd(d,2))}{2d\omega}. $$
Both cases occur.
\end{proposition}

\begin{proof}
Let $\alpha G_\Omega$ be a generator of $G/G_\Omega$. If $d=1$, the claim is trivial.
Suppose $d=2$.
Since $|G_\Omega|$ is odd, $G$ is cyclic of order $2d\mid(q+1)$ and fixes pointwise a self-polar triangle $\{P,Q,R\}\subset\PG(2,q^2)\setminus\cH_q$.
Supppose $d>2$.
Then $\alpha$ is of type (B2) and $G$ is cyclic.
The number of elements of type (A) is either $\omega-1$ or $2\omega-1$, according to $d$ odd or $d$ even, respectively; any other nontrivial element is of type (B2).

The claim follows by direct computation with Theorem \ref{caratteri}.
\end{proof}

\begin{proposition}
Let $G\leq\cM_q$ be such that $G/G_\Omega\cong Dic_n$, with $1<n\mid\frac{q-1}{2}$. Then
$$g(\cH_q/G)=\frac{(q-1)(q+1-2\omega)}{8n\omega}.$$
\end{proposition}

\begin{proof}
By  Remark \ref{rimarco} and Proposition \ref{sottogruppiSU}, $G=G_H\times G_\Omega$ with $G_H\cong Dic_n$.
Any nontrivial element $\sigma\in G$ is of type (A) if $\sigma\in\langle\iota\rangle\times G_\Omega$, and of type (B2) otherwise.
The claim follows by Theorem \ref{caratteri}.
\end{proof}

\begin{proposition}\label{sl2p^k}
Let $G\leq\cM_q$ be such that $G/G_\Omega\cong\SL(2,p^k)$ with $k\mid n$, $r=n/k$. Then
$$ g(\cH_q/G)=1+\frac{q^2-q-2-\Delta}{2p^k(p^{2k}-1)\omega},$$
where
$$ \Delta= (p^{2k}-1)(q+2)+p^{2k}-1 +q+1+p^k(p^k+1)(p^k-3)\omega + p^k(p^k-1)^2(\gcd(r,2)-1) + 2(p^{2k}-1)(\omega-1) $$
$$ +2(\omega-1)(q+1) +  p^k(p^k-1)^2(\omega-1)(\gcd(r,2)-1) + (\gcd(\omega,p^k+1)-1)p^k(p^k-1)(q+1)(2-\gcd(r,2)). $$
\end{proposition}

\begin{proof}
By Remark \ref{rimarco}, $G=G_H\times G_\Omega$ with $G_H\cong\SL(2,p^k)$.
Clearly $(p^k-1)\mid(q-1)$, while $(p^k+1)\mid(q-1)$ or $(p^k+1)\mid(q+1)$ according to $r$ even or $r$ odd, respectively.
The nontrivial elements in $G_H$ are classified in the proof of \cite[Proposition 4.3]{MZ2} as follows:
\begin{itemize}
\item $p^{2k}-1$ elements of order $p$ and type (C);
\item $p^{2k}-1$ elements of order $p$ times a nontrivial divisor of $q+1$, which are of type (E);
\item $1$ involution $\iota$, of type (A);
\item $\frac{p^k(p^k+1)(p^k-3)}{2}$ elements of order a divisor of $p^k-1$ different from $2$, which are of type (B2);
\item $\frac{p^k(p^k-1)^2}{2}$ elements of order a divisor of $p^k+1$ different from $2$, which are of type (B1) or (B2) according to $r$ odd or $r$ even, respectively.
\end{itemize}
The product of an element $\sigma\in G_H$ by a nontrivial element $\tau\in G_\Omega$ is as follows.
If $\sigma$ is of type (C) or (E), then $\sigma\tau$ is of type (E).
If $\sigma\in\langle\iota\rangle$, then $\sigma\tau$ is of type (A).
If $\sigma$ is of type (B2), then $\sigma\tau$ is of type (B2).

If $\sigma$ is of type (B1) and $o(\sigma)\mid|G_\Omega|$, then $G_\Omega$ contains exactly $2$ elements $\tau_1,\tau_2$ such that $\sigma\tau_1,\sigma\tau_2$ are of type (A); otherwise, $\sigma\tau$ is of type (B1).
If $r$ is even, there are no elements of type (B1). Assume $r$ odd. The elements of type (B1) together with $\langle\iota\rangle$ form $\frac{p^k(p^k-1)}{2}$ cyclic groups of order $p^k+1$ which intersect pairwise in $\langle\iota\rangle$. Then, the number of elements $\sigma\tau$ of type (A) with $\sigma$ of type (B1) is $\frac{p^k(p^k-1)}{2}\cdot(\gcd(\omega,p^k+1)-1)\cdot2$.

Now the claim follows by direct computation with Theorem \ref{caratteri}.
\end{proof}

\begin{proposition}
Let $G\leq\cM_q$ be such that $G/G_\Omega\cong{\rm TL}(2,p^k)$, where $k\mid n$ and $n/k$ is even. Then
$$ g(\cH_q/G)=1+\frac{q^2-q-2-\Delta}{4p^k(p^{2k}-1)\omega}, $$
where
$$ \Delta= (p^{2k}-1)(q+2)+p^{2k}-1 +q+1+p^k(p^k+1)(p^k-3)\omega + p^k(p^k-1)^2 + 2(p^{2k}-1)(\omega-1) $$
$$ +2(\omega-1)(q+1) +  p^k(p^k-1)^2(\omega-1)+2p^k(p^{2k}-1)\omega. $$
\end{proposition}

\begin{proof}
By  Remark \ref{rimarco} and Proposition \ref{sottogruppiSU}, $G=G_H \times G_\Omega$ with $G_H =\langle L,\delta\rangle\cong{\rm TL}(2,p^k)$, $L\cong\SL(2,p^k)$, $o(\delta)=2(p^k-1)$.
The nontrivial elements in $L\times G_\Omega$ are already classified according to their type in the proof of Proposition \ref{sl2p^k}.
 Every element in $G_H\setminus L$ is of type (B2); see the proof of \cite[Proposition 4.4]{MZ2}.
Hence, for every $\sigma\in G_H\setminus L$ and $\tau\in G_\Omega$, $\sigma\tau$ is of type (B2).
The claim now follows by direct computation with Theorem \ref{caratteri}.
\end{proof}

The case $G/G_{\Omega}$ isomorphic to a cyclic subgroup of ${\rm SU}^{\pm}(2,q)$ of order $2$ not in $H$ has already been considered in Proposition \ref{ciclicoq-1}.

\begin{proposition}
Let $G\leq\cM_q$ be such that $G/G_\Omega\cong\SL(2,3)\rtimes C_2\cong SmallGroup(48,29)$, with $p\geq5$ and $8\nmid(q-1)$. Then
$$ g(\cH_q/G)=\frac{(q+1)(q-2\omega-13)+60\omega-16r}{96\omega},
\quad\textrm{where}\quad
r=\begin{cases} 2\omega & \textrm{if}\quad3\mid(q-1), \\
0 & \textrm{if}\quad3\mid(q+1),\;3\nmid\omega, \\
q+1 & \textrm{if}\quad3\mid\omega. \end{cases} $$
\end{proposition}

\begin{proof}
By Remark \ref{rimarco} and Proposition \ref{sottogruppiSU}, $G=G_{\pm}\times G_\Omega$ with $G_\pm \cong SmallGroup(48,29)$ and $G_H\cong\SL(2,3)$.
By Lemma \ref{classificazione}, $G_{\pm}$ contains $13$ involutions, of type (A); $18$ elements of order $8$ or $4$, of type (B2); $16$ elements of order $6$ or $3$, which are of type (B2) or (B1) according to $3\mid(q-1)$ or $3\mid(q+1)$.
The element $\sigma\tau$, where $\sigma\in G_\pm$ and $\tau\in G_\Omega\setminus\{id\}$, is as follows.
If $\sigma$ is the unique involution $\iota$ of $G_H$, then $\sigma\tau$ is of type (A).
If $\sigma$ is an involution different from $\iota$, then $\sigma\tau$ is of type (B1).
If $\sigma$ is of type (B2), then $\sigma\tau$ is of type (B2).
If $\sigma$ has order $6$ and is of type (B1), then $\sigma\tau$ is of type (B1).
If $\sigma$ has order $3$ and is of type (B1), then there are $\gcd(|G_\Omega|,3)-1$ elements of $G_\Omega$ such that $\sigma\tau$ is of type (A), while for any other $\tau\in G_\Omega$ $\sigma\tau$ is of type (B1); in fact, $\sigma\tau$ is of type (A) if and only if $o(\tau)=3$.

The claim follows by the Riemann-Hurwitz formula and Theorem \ref{caratteri}.
\end{proof}

\begin{proposition}
Let $G\leq\cM_q$ be such that $G/G_\Omega$ is cyclic of order $2d>2$, where either $d\mid(q-1)$ and $d\nmid\frac{q-1}{2}$, or $d\mid(q+1)$. 
Assume also that $G$ does not stabilize any self-polar triangle $T\subset\PG(2,q^2)\setminus\cH_q$.
Then either $d\mid(q-1)$ and $d\nmid\frac{q-1}{2}$, or $d=2$; in both cases,
\begin{equation}\label{ciclico2d}
g(\cH_q/G)=\frac{(q-1)(q+1-2\omega)}{4d\omega}.
\end{equation}
Whenever $d$ satisfies the numerical assumptions, a subgroup $G\leq\cM_q$ with $g(\cH_q/G)$ as in Equation \eqref{ciclico2d} exists.
\end{proposition}

\begin{proof}
Let $\alpha G_\Omega$ be a generator of $G/G_\Omega$.

Suppose that $d\mid(q-1)$ and $d\nmid\frac{q-1}{2}$. Then $o(\alpha)\mid(q^2-1)$ and $\alpha$ is of type (B2). The group $G$ fixes pointwise the $2$ fixed points $Q,R\in\ell\cap\cH_q$ other than $P$. Hence, $G$ is cyclic; we can assume that $G=\langle\alpha\rangle$. Since $\gcd(o(\alpha),q+1)=2\omega$, $G$ contains $2\omega-1$ elements of type (A) and $2d\omega-2\omega$ elements of type (B2); Equation \eqref{ciclico2d} follows. Such a group $G$ does exist in $\cM_q$, being generated by any element of type (B2) and order $2d\omega$.

Suppose that $d\mid(q+1)$. If $2d\mid(q+1)$, then Lemma \ref{gianoti} implies that $G$ stabilizes a self-polar triangle $T\subset\PG(2,q^2)\setminus\cH_q$; hence, we can assume that $2d\nmid(q+1)$, that is, $d$ is even and $d/2$ is odd. Then $4\mid o(\alpha)$ and $\alpha$ is of type (B2).
Also, $o(\alpha^2)\mid(q+1)$ implies that $\alpha^4\in G_\Omega$, so that $d=2$.
As above, $G$ is cyclic, and we can assume that $\alpha$ is a generator of $G$.
The group $G$ has $2\omega-1$ elements of type (A) and $4\omega-2\omega$ elements of type (B2); Equation \eqref{ciclico2d} follows.
Such a group $G$ does exist in $\cM_q$, being generated by any element of type (B2) and order $4\omega$.
\end{proof}

\begin{proposition}
Let $G\leq\cM_q$ be such that $G/G_\Omega$ is dihedral of order $2d$, with $2<d\mid(q-1)$. Then
$$ g(\cH_q/G)=\frac{(q+1)(q-1-\gcd(d,2)\cdot\omega-d)+2\omega(d+\gcd(d,2))}{4d\omega}. $$
\end{proposition}

\begin{proof}
By Remark \ref{rimarco}, $G=G_\pm \times G_\Omega$, where $G_\pm=\langle\alpha G_\Omega,\gamma G_\Omega\rangle$ is dihedral; we can assume that $o(\alpha)=d$ and $o(\gamma)=2$.
Such a group $G$ actually exists in $\cM_q$, as shown in the proof of Proposition \ref{sottogruppiSU} using the model \eqref{M2} of $\cH_q$.
The nontrivial elements of $G$ are as follows:
$\gcd(d,2)\cdot\omega-1$ elements of type (A) and center $P$ in $G_\Omega$ or $\langle\iota\rangle\times G_\Omega$, according to $d$ odd or $d$ even;
$d$ elements of type (A) with center on $\ell$, in $G_\pm$;
$(d-\gcd(d,2))\omega$ elements of type (B2) in $\langle\alpha\rangle\times G_\Omega$;
 $d(\omega-1)$ elements of type (B1) as the product of an involution in $G_\pm\setminus\langle\iota\rangle$ by an element of $G_\Omega$.
The claim follows by direct computation with Theorem \ref{caratteri}.
\end{proof}

\begin{proposition}
Let $G\leq\cM_q$ be such that
$G/G_\Omega\cong \hat{Dic}_m$,
where $m\mid\frac{q-1}{2}$ and $m\nmid\frac{q-1}{4}$. Then
$$ g(\cH_q/G)=1+\frac{q^2-q-2-[(2m+2\omega-1)(q+1)+4\omega(3m-1)]}{16m\omega}. $$
\end{proposition}

\begin{proof}
By Remark \ref{rimarco}, $G=G_{\pm}\times G_\Omega$ with $G_\pm\cong \hat{Dic}_m$.
From the proof of Proposition \ref{sottogruppiSU}, the nontrivial elements $G_{\pm}$ are exactly: $1$ involution $\iota$; $2m$ other involutions of type (A) with center on $\ell$; $6m-2$ elements of type (B2).
The nontrivial elements of $G_\Omega$ are of type (A).
The product $\sigma\tau$ with $\sigma\in G_{\pm}$ and $\tau\in G_{\Omega}\setminus\{id\}$ is as follows: of type (A), if $\sigma=\iota$; of type (B1), if $\sigma$ is an involution different from $\sigma$; of type (B2), if $\sigma$ is of type (B2).
The claim follows by direct computation with Theorem \ref{caratteri}.
\end{proof}

\begin{proposition}
Let $G\leq \cM_q$ be such that $G/G_\Omega\cong {\rm SU}^{\pm}(2,p^k)\cong \SL(2,p^k)\rtimes C_2$, where $k\mid n$ and $n/k$ is odd. Then
$$ g(\cH_q/G)=1+\frac{q^2-q-2-\Delta}{4p^k(p^{2k}-1)\omega}, $$
where
$$ \Delta= (q+1)+p^k(p^k+1)(p^k-3)+(p^{2k}-1)(q+3) + p^k(p^k-1)(q+1) + p^k(p^{2k}-1) +(2\omega-2)(q+1) $$
$$  + 2(p^{2k}-1)(\omega-1) + 2p^k(p^k+1)(p^k-2)(\omega-1) + 2p^k(p^k-1)(q+1)\left(\gcd(p^k+1,\omega)-1\right). $$
\end{proposition}

\begin{proof}
By Remark \ref{rimarco} and Proposition \ref{sottogruppiSU}, $G=G_{\pm}\times G_\Omega$ with $G_\pm \cong {\rm SU}^{\pm}(2,p^k)$ and $G_H\cong\SL(2,p^k)$.

The nontrivial elements of $G_H$ are classified according to their type in the proof of \cite[Proposition 4.3]{MZ2}. Namely, $G_H$ contains exactly:
$1$ element $\iota$ of type (A);
$\frac{p^k(p^k-1)^2}{2}$ elements of type (B1), forming $\frac{p^k(p^k-1)}{2}$ cyclic groups of order $p^k+1$ with pairwise intersection $\langle\iota\rangle$;
$\frac{p^k(p^k+1)(p^k-3)}{2}$ elements of type (B2), forming $\binom{p^k+1}{2}$ cyclic groups of order $p^k-1$ with pairwise intersection $\langle\iota\rangle$;
$p^{2k}-1$ elements of type (C), forming $p^k+1$ elementary abelian groups of order $p^k$ with trivial pairwise intersection;
$p^{2k}-1$ elements of type (E), contained in $p^k+1$ cyclic groups of order $2p^k$ with pairwise intersection $\langle\iota\rangle$.

The elements of $G_\pm \setminus G_H$ are classified as follows.
\begin{itemize}
\item $G_\pm \setminus G_H$ contains exactly $p^k(p^k-1)$ elements of type (A) and $\frac{p^k(p^k-1)^2}{2}$ elements of type (B1).

In fact, let $\alpha\in\cM_q\setminus(\langle\iota\rangle\times\Omega)$ be of type (A) or (B1), and $\{Q,R\}\in\PG(2,p^{2k})\setminus\cH_q$ be the fixed points of $\alpha$ on $\ell$.
Let $\cH_q$ have equation \eqref{M1}; up to conjugation, $G_{\pm}$ is made by $\mathbb{F}_{p^{2k}}$-rational elements, as pointed out in Equation \eqref{sottoSUnelsottocampo}.
Also, up to conjugation, $P=(0,0,1)$, $Q=(1,0,0)$, and $R=(0,1,0)$, so that $\alpha=\diag(\lambda,\mu,1)$ with $\lambda^{o(\alpha)}=\mu^{o(\alpha)}=1$. For any $(p^k+1)$-th root of unity $\lambda$, $\alpha\in G_\pm \setminus G_H$ if and only if $\det(\alpha)=-1$, i.e. $\mu=-\lambda^{-1}$. Note that $-\lambda^{-1}\ne\lambda$ since $4\nmid(p^k+1)$.
Hence, after the choice of $\{Q,R\}$, there are exactly $2$ elements of type (A) (namely, $\alpha=\diag(1,-1,1)$ and $\alpha=\diag(-1,1,1)$) and $p^k-1$ elements of type (B1) (namely, $\alpha=\diag(\lambda,-\lambda^{-1},1)$ with $\lambda\ne\pm1$).

There are exactly $\frac{p^{2k}+1-(p^k+1)}{2}=\frac{p^k(p^k-1)}{2}$ choices for $\{Q,R\}$.
In fact, $Q$ can be chosen as anyone of the $\mathbb{F}_{p^{2k}}$-rational points of $\ell$ which are not on $\cH_q$; then $R$ is uniquely determined as the pole of the line $PQ$ with respect to the unitary polarity associated to $\cH_q(\fqs)$.

\item As $p\nmid[G:G_H]$, there are no $p$-elements in $G\setminus G_H$.
Also, any element of type (E) in $\cM_q$ is the product $\sigma\tau$ of a $p$-element of type (C) by an element $\tau$ of type (A) in $\langle\iota\rangle\times\Omega$ (i.e., $\tau$ has center $P$).
Since $G_\pm\cap(\langle\iota\rangle\times \Omega)=\langle\iota\rangle$, there are no elements of type (E) in $G\pm \setminus G_H$.

\item Any other element of $G_\pm \setminus G_H$ is of type (B2); their number is
$$p^k(p^{2k}-1)-p^k(p^k-1)-\frac{p^k(p^k-1)^2}{2}=\frac{p^k(p^k-1)(2p^k+2-2-p^k+1)}{2}=\frac{p^k(p^{2k}-1)}{2}.$$
\end{itemize}

The nontrivial elements of $G_\Omega$ are of type (A).
The elements $\sigma\tau$ with $\sigma\in G_\pm\setminus\{id\}$ and $\tau\in G_\Omega\setminus\{id\}$ are classified as follows.
\begin{itemize}
\item If $\sigma$ is of type (C) or (E), then $\sigma\tau$ is of type (E).
\item If $\sigma=\iota$, then $\sigma\tau$ is of type (A).
\item If $\sigma$ is an involution different from $\iota$, then $\sigma\tau$ is of type (B1).
\item If $\sigma$ is of type (B2), then $\sigma\tau$ is of type (B2).
\item
Let $\sigma$ be of type (B1) and $\{Q,R\}$ be the points fixed by $\sigma$ on $\ell$; we have $\frac{p^k(p^k-1)}{2}$ choices for $\{Q,R\}$. Arguing as above, we can use the model \eqref{M1} for $\cH_q$, assume that $\{P,Q,R\}$ is the fundamental triangle, and that $\sigma=\diag(\lambda,\lambda^{-1},1)$ or $\sigma=\diag(\lambda,-\lambda^{-1},1)$, with $\lambda^{p^k+1}=1$, $\lambda\ne\pm1$.
If $o(\lambda)\mid\omega$, then there exists exactly one $\tau\in G_\Omega \setminus\{id\}$ such that $\sigma\tau$ is of type (A); otherwise, $\sigma\tau$ is of type (B1).
Altogether, when $\sigma$ ranges overe the elements of type (B1), there are exactly $\frac{p^k(p^k-1)}{2}\cdot(\gcd(\omega,p^k+1)-1)\cdot4$ elements $\sigma\tau$ of type (A); the other ones are of type (B1).
\end{itemize}
Now the claim follows by direct computation with Theorem \ref{caratteri}.
\end{proof}

\section{The complete list of genera of quotients of $\cH_q$ for $q \equiv 1 \pmod 4$}\label{sec: lista}

This section provides the explicit complete list of genera of quotients $\cH_q/G$, $G \leq \PGU(3,q)$ for $q=p^n \equiv 1 \pmod 4$, and hence it gives a collection of the results obtained in this paper together with the results already obtained in the literature.  It will be divided into subsections corresponding to the maximal subgroup of $\PGU(3,q)$ containing $G$. The case in which $G \leq \PGU(3,q)_Q$, $Q \in \PG(2,q^2) \setminus \cH_q$ will not be repeated here as it was already described in the previous sections. 

\subsection{$G \leq \PGU(3,q)_P$, $P \in \cH_q(\mathbb{F}_{q^2})$}
\ \\ \\
Let $P \in \cH_q(\mathbb{F}_{q^2})$ and suppose that $G \leq \PGU(3,q)_{P}$, the stabilizer of $P$ in $\aut(\cH_q)=\PGU(3,q)$. Since $\PGU(3,q)$ acts transitively on $\cH_q(\mathbb{F}_{q^2})$ we can assume that $P=P_\infty$, that is, the unique point at infinity of the model $y^{q+1}=x^q+x$ of $\cH_q$. 
The complete list for odd values of $q$ of genera of quotients $\cH_q/G$ where $G \leq \PGU(3,q)_{P_\infty}$ was determined in \cite{BMXY} and partially in \cite{GSX}. 

In these papers, an element $\sigma \in \PGU(3,q)_{P_\infty}$ is uniquely described as a triple of elements in $\mathbb{F}_{q^2}$, $\sigma=[a,b,c]$. The authors associate to $G$ three sets, namely
$$G_1=\{a \mid [a,b,c] \in G\}, \quad G_2=\{b \mid [1,b,c] \in G\}, \quad G_3=\{c \mid [1,0,c] \in G\}$$
with $|G_1|=g_1$, $|G_2|=p^{g_2}$, and $|G_3|=p^{g_3}$. 

In \cite{GSX} it is proved that the genus of the quotient $\cH_q/G$ depends just on the values $g_i$. Namely, the following theorem is proved.

\begin{theorem}{{\rm(}see \cite{GSX} and \cite[Lemma 4.1]{BMXY}{\rm )}} \label{stabinf}
Let $G \leq \PGU(3,q)_{P_\infty}$, $G_i$, and $g_i$ be defined as above, and let $d=\gcd(g_1,q+1)$. Then
\begin{equation}\label{puntofisso}
g(\cH_q/G)=\frac{q-p^{g_3}}{2|G|}(q-(d-1)p^{g_2}).
\end{equation}
Defining $r$ and $u$ to be the smallest integers such that $p^r \equiv 1 \pmod{g_1}$ and $p^u \equiv 1 \pmod{g_1/d}$ respectively, then $|G|=h_1 p^{g_2 r + g_3 u}$. 
\end{theorem}

At this point the authors provide necessary and sufficient conditions to the triple $(g_1,g_2,g_3)$ to be associated with an existing subgroup $G$ of $\PGU(3,q)_{P_\infty}$, so that Theorem \ref{stabinf} gives the complete list of genera when $G$ is in $\PGU(3,q)_P$ for some $P \in \cH_q(\mathbb{F}_{q^2})$. The authors observed that when $g_1 \mid (q^2-1)$ is fixed it is sufficient to consider the cases $0 \leq g_2 \leq 2n/r$ and $0 \leq g_3 < n/u$ as otherwise $\cH_q/G$ would be rational.

\begin{theorem}{{\rm(}\cite[Theorems 4.3 and 5.8]{BMXY}{\rm )}}
Define $r$ and $u$ as above. Let $M_G(p,n)=\{(g_1,g_2,g_3) : g_1 \mid (q^2-1), \ 0 \leq g_2 \leq 2n/r, \ and \ 0 \leq g_3 < n/u \ for \ some \ G \leq \PGU(3,q)_{P_\infty}\}$. Then
\begin{itemize}
\item If $g_1 \mid (q-1)$ then for every $0 \leq g_2 \leq n/r$ and $0 \leq i < n/u $, $(g_1,g_2,i) \in M_G(p,n)$.
\item If $g_1 \mid (q^2-1)$ but $g_1 \nmid (q-1)$ then for every $0 \leq i < n/u $, $(g_1,0,i) \in M_G(p,n)$.
\item If $g_1 \mid (q^2-1)$ but $g_1 \nmid (q-1)$ then for every $0 < g_2 \leq 2n/r$ and $\mu_{g_2} \leq i < n/u$, $(g_1,g_2,i) \in M_G(p,n)$, where
$$\mu_{g_2}=\frac{r}{2u} \min_{\ell \mid (2n/r)} \bigg( 2 \bigg\lceil \frac{g_2}{\ell}\bigg\rceil -1\bigg)\ell.$$
\end{itemize}
Moreover, this is the complete list of elements in $M_G(p,n)$.
\end{theorem}

\subsection{$G \leq \PGU(3,q)_T$, $T$ a self-polar triangle in $\PG(2,q^2)\setminus\cH_q$}
\ \\ \\ 
Let $T=\{P_1,P_2,P_3\}$ be a self-polar triangle in $\PG(2,q^2)\setminus\cH_q$ and $\PGU(3,q)_T$ be the maximal subgroup of $\PGU(3,q)$ stabilizing $T$.
We have $\PGU(3,q)_T=(C_{q+1}\times C_{q+1})\rtimes \mathbf{S}_3$, where $C_{q+1}\times C_{q+1}$ stabilizes $T$ pointwise and $\mathbf{S}_3$ acts faithfully on $T$.

The genera of quotients $\cH_q/G$ where $G\leq\PGU(3,q)_T$ are completely classified in \cite{DVMZ} as follows according to the action of $G$ in $T$.
We define $G_T=G\cap(C_{q+1}\times C_{q+1})$.

\begin{theorem}{{\rm(}\cite[Theorem 3.1]{DVMZ}{\rm )}}\label{fissatorepuntuale}
Let $q+1=\prod_{i=1}^{\ell}p_i^{r_i}$ be the prime factorization of $q+1$.
\begin{itemize}
\item[(i)] For any divisors $a=\prod_{i=1}^{\ell}p_i^{s_i}$ and $b=\prod_{i=1}^{\ell}p_i^{t_i}$ of $q+1$ $($$0\leq s_i,t_i\leq r_i$$)$, let $c=\prod_{i=1}^{\ell}p_i^{u_i}$ be such that, for all $i=1,\ldots,\ell$, we have $u_i=\min\{s_i,t_i\}$ if $s_i\ne t_i$, and $s_i\leq u_i\leq r_i$ if $s_i=t_i$.
Define $d=a+b+c-3$.
Let $e=\frac{abc}{\gcd(a,b)}\cdot\prod_{i=1}^{\ell}p_i^{v_i}$, where for all $i$'s $v_i$ satisfies $0\leq v_i\leq r_i-\max\{s_i,t_i,u_i\}$.
We also require that, if $p_i=2$ and either $2\nmid abc$ or $2\mid\gcd(a,b,c)$, then $v_i=0$.
Then there exists a subgroup $G$ of $C_{q+1}\times C_{q+1}$ such that
\begin{equation}\label{generefissatore}
 g(\cH_q/G)=\frac{(q+1)(q-2-d)+2e}{2e}.
\end{equation}
\item[(ii)] Conversely, if $G\leq C_{q+1}\times C_{q+1}$, then the genus of $\cH_q/G$ is given by Equation \eqref{generefissatore}, where $e=|G|$ and $d$ is the number of elements of type {\rm (A)} in $G$; $d,e$ satisfy the numerical assumptions in point {\it (i)}, for some $a,b,c$.
\end{itemize}
\end{theorem}

\begin{proposition}{{\rm(}\cite[Proposition 3.4]{DVMZ}{\rm )}}\label{indice2dispari}
Let $q$ be odd.
\begin{itemize}
\item[(i)] Let $\ell$, $a$, $c$, and $e$ be positive integers satisfying $e\mid(q+1)^2$, $c\mid(q+1)$, $\ell\mid c$, $a\mid c$, $ac\mid e$, $\frac{e}{a}\mid(q+1)$, and $\gcd(\frac{e}{ac},\frac{c}{a})=1$.
If $2\mid a$ or $2\nmid c$, we also require that $2\nmid \frac{e}{ac}$.
Then there exists a subgroup $G\leq (C_{q+1}\times C_{q+1})\rtimes \mathbf{S}_3$ of order $2e$ such that $|G\cap(C_{q+1}\times C_{q+1})|=e$ and 
\begin{equation}\label{genereindice2dispari}
g(\cH_q/G)=\frac{(q+1)\left(q-2a-c+1-h\right)-2k + 4e}{4e},
\end{equation}
where
$$ (h,k)=\begin{cases}
\left(\frac{e}{c},\frac{e}{2}\right) & \qquad \textrm{if} \qquad 2a\nmid (q+1) \,;\\
\left(\frac{e}{c},0\right) & \qquad \textrm{if} \qquad 2a \mid (q+1),\; 2a\nmid c\;; \\
\left(0,e\right) & \qquad \textrm{if} \qquad 2a\mid c,\;2\ell\nmid(q+1)\;; \\
\left(0,0\right) & \qquad \textrm{if} \qquad 2a\mid c,\; 2\ell\mid (q+1),\; 2\ell\nmid c\;; \\
\left(\frac{2e}{c},0\right) & \qquad \textrm{if} \qquad 2a\mid c,\; 2\ell\mid c\;. \\
\end{cases} $$
\item[(ii)] Conversely, if $G\leq(C_{q+1}\times C_{q+1})\rtimes \mathbf{S}_3$ and $G\cap(C_{q+1}\times C_{q+1})$ has index $2$ in $G$, then the genus of $\cH_q/G$ is given by Equation \eqref{genereindice2dispari}, where: $e=|G|/2$; without loss of generality, $a-1$ is the number of homologies in $G$ with center $P_1$ which is equal to the number of homologies in $G$ with center $P_2$, and $c-1$ is the number of homologies in $G$ with center $P_3$; $\ell=\frac{o(\beta)}{2}$ for some $\beta\in G\setminus G_T$; $\ell,a,c,e$ satisfy the numerical assumptions in point {\it (i)}.
\end{itemize}
\end{proposition}

\begin{proposition}{{\rm(}\cite[Proposition 3.5]{DVMZ}{\rm )}}\label{indice3caso1}
Let $q$ be such that $3\nmid(q+1)$.
\begin{itemize}
\item[(i)]
Let $a$ and $e$ be positive integers satisfying $e\mid(q+1)^2$, $a^2\mid e$, $\frac{e}{a}\mid(q+1)$, $2\nmid\frac{e}{a^2}$, and $\gcd(\frac{e}{a^2},a)=1$.
We also require that there exists a positive integer $m\leq \frac{e}{a^2}$ such that $\frac{e}{a^2}\mid(m^2-m+1)$.
Then there exists a subgroup $G\leq(C_{q+1}\times C_{q+1})\rtimes \mathbf{S}_3$ of order $3e$ such that $|G\cap(C_{q+1}\times C_{q+1})|=e$ and
\begin{equation}\label{genereindice3caso1}
g(\cH_q/G)=\frac{(q+1)(q-3a+1)+2e}{6e}.
\end{equation}
\item[(ii)] Conversely, if $G\leq(C_{q+1}\times C_{q+1})\rtimes \mathbf{S}_3$ and $G\cap(C_{q+1}\times C_{q+1})$ has index $3$ in $G$, then the genus of $\cH_q/G$ is given by Equation \eqref{genereindice3caso1}, where: $e=|G|/3$; the number of homologies in $G$ with center $P_i$ is $a-1$ for $i=1,2,3$; there exists $m$ such that $a,e,m$ satisfy the numerical assumptions in point {\it (i)}.
\end{itemize}
\end{proposition}

\begin{proposition}{{\rm(}\cite[Proposition 3.6]{DVMZ}{\rm )}}\label{indice3caso2}
Let $q$ be such that $3\mid(q+1)$.
\begin{itemize}
\item[(i)]
Let $a$, $e$, and $\ell$ be positive integers satisfying $e\mid(q+1)^2$, $a^2\mid e$, $\frac{e}{a}\mid(q+1)$, $2\nmid\frac{e}{a^2}$, $\gcd(\frac{e}{a^2},a)=1$, and $\ell\mid(q+1)$.
We also require that there exists a positive integer $m\leq \frac{e}{a^2}$ such that $\frac{e}{a^2}\mid(m^2-m+1)$.
Then there exists a subgroup $G\leq(C_{q+1}\times C_{q+1})\rtimes \mathbf{S}_3$ of order $3e$ such that $|G\cap(C_{q+1}\times C_{q+1})|=e$ and
\begin{equation}\label{genereindice3caso2}
g(\cH_q/G)=\frac{(q+1)(q-3a+1)+h\cdot e}{6e}, \qquad\textrm{with}\qquad h=\begin{cases}
2 & \textrm{if}\quad a\nmid\frac{q+1}{3}, \\
0 & \textrm{if}\quad a\mid\frac{q+1}{3},\;\ell\nmid\frac{q+1}{3}, \\
6 & \textrm{if}\quad a\mid\frac{q+1}{3},\;\ell\mid\frac{q+1}{3}. \\
\end{cases}
\end{equation}
\item[(ii)] Conversely, if $G\leq(C_{q+1}\times C_{q+1})\rtimes \mathbf{S}_3$ and $G\cap(C_{q+1}\times C_{q+1})$ has index $3$ in $G$, then the genus of $\cH_q/G$ is given by Equation \eqref{genereindice3caso2}, where: $e=|G|/3$; the number of homologies in $G$ with center $P_i$ is $a-1$ for $i=1,2,3$; there exist $\ell$ and $m$ such that $a,e,\ell,m$ satisfy the numerical assumptions in point {\it (i)}.
\end{itemize}
\end{proposition}

\begin{proposition}{{\rm(}\cite[Proposition 3.7]{DVMZ}{\rm )}}\label{indice6}
\begin{itemize}
\item[(i)] Let $a$ be a divisor of $q+1$.
We choose $e=a^2$ if $3\nmid(q+1)$ or $3\mid a$; $e\in\{a^2,3a^2\}$ if $3\mid(q+1)$ and $3\nmid a$.
Then there exists a subgroup $G\leq(C_{q+1}\times C_{q+1})\rtimes \mathbf{S}_3$ of order $6e$ such that $|G\cap(C_{q+1}\times C_{q+1})|=e$ and
\begin{equation}\label{genereindice6}
g(\cH_q/G)=\frac{(q+1)(q-3a+1-\frac{3e}{a})-2r-3s+12e}{12e},
\end{equation}
where
$$
r=\begin{cases} \frac{7e}{2} & \textrm{if}\quad q\equiv0\,\textrm{ or }\,1\!\!\!\pmod3\;\textrm{and}\;a\nmid\frac{q+1}{2}, \\
\frac{3e}{2} & \textrm{if}\quad q\equiv2\!\!\!\pmod3\;\textrm{and}\;a\nmid\frac{q+1}{2}, \\
2e & \textrm{if}\quad q\equiv0\,\textrm{ or }\,1\!\!\!\pmod3\;\textrm{and}\;a\mid\frac{q+1}{2}, \\
0 & \textrm{if}\quad q\equiv2\!\!\!\pmod3\;\textrm{and}\;a\mid\frac{q+1}{2}; \end{cases}
\quad\textrm{and}\quad
s= \begin{cases} \frac{4e}{3} & \textrm{if}\quad q\equiv2\!\!\!\pmod3\;\textrm{ and }\; a\nmid\frac{q+1}{3}, \\
0 & \textrm{otherwise}. \end{cases}
$$
\item[(ii)] Conversely, if $G\leq(C_{q+1}\times C_{q+1})\rtimes \mathbf{S}_3$ and $G\cap(C_{q+1}\times C_{q+1})$ has index $6$ in $G$, then the genus of $\cH_q/G$ is given by Equation \eqref{genereindice6}, where: $e=|G|/6$; the number of homologies in $G$ with center $P_i$ is $a-1$ for $i=1,2,3$; $a$ and $e$ satisfy the numerical assumptions in point {\it (i)}.
\end{itemize}
\end{proposition}

\subsection{$G \leq \PGU(3,q)_T$, $T$ a triangle in $\cH_q(\mathbb{F}_{q^6})\setminus\cH_q(\mathbb{F}_{q^2})$}
\ \\ \\ 
Let $T=\{P_1,P_2,P_3\}$ be a triangle in $\cH_q(\mathbb{F}_{q^6})\setminus\cH_q(\mathbb{F}_{q^2})$ which is invariant under the Frobenius collineation $\Phi_{q^2}$ of $\PG(2,q^6)$, and $\PGU(3,q)_T$ be the maximal subgroup of $\PGU(3,q)$ stabilizing $T$.
We have $\PGU(3,q)_T = S \rtimes C_3$, where $S\cong C_{q^2-q+1}$ is a Singer subgroup stabilizing $T$ poitwise and acting semiregularly on $\PG(2,q^2)$; $C_3$ acts faithfully on $T$.

The genera of quotients $\cH_q/G$ where $G\leq\PGU(3,q)_T$ are completely classified in \cite{CKT2} whenever $p\nmid|G|$.
The case $p\mid|G|$ happens only if $p=3$.
This case was not considered in \cite{CKT2}, and ca be dealt with using Theorem \ref{caratteri} and the fact that each element in $G\setminus G_T$ has order $3$.

\begin{theorem}{{\rm(}\cite[Proposition 4.2]{CKT2}{\rm )}}\label{singer}
\begin{itemize}
\item[(i)] Let $\nu\mid(q^2-q+1)$.
Then there exists $G\leq\PGU(3,q)_T$ such that $|G_T|=\nu$ and one of the following holds:
\begin{equation}\label{casisinger}
g(\cH_q/G)=
\begin{cases}
\frac{1}{2}\left(\frac{q^2-q+1}{\nu}-1\right); \\
\frac{q^2-q+1-\nu}{6\nu},\quad q\equiv2\pmod3\quad\textrm{and}\quad3\mid\nu,\quad\textrm{or}\quad q\equiv0,1\pmod3; \\
\frac{q^2-q+1-3\nu}{6\nu},\quad q\equiv2\pmod3\quad\textrm{and}\quad3\nmid\nu; \\
\frac{q^2-q+1+3\nu}{6\nu},\quad q\equiv2\pmod3\quad\textrm{and}\quad3\nmid\nu. 
\end{cases}
\end{equation}
When $q\equiv2\pmod3$ and $3\nmid\nu$, then both the third and the fourth line in Equation \eqref{casisinger} are obtained for some $G\leq\PGU(3,q)$.
\item[(ii)] Conversely, let $G\leq\PGU(3,q)_T$. If $G=G_T$, then $g(\cH_q/G)$ is given by the first line in Equation \eqref{casisinger} with $\nu=|G|$.
If $G\ne G_T$, then $g(\cH_q/G)$ is given by the second or third or fourth line in Equation \eqref{casisinger} with $\nu=|G|/3$.
\end{itemize}
\end{theorem}

\subsection{$G\leq\PGU(3,q)$ has no fixed points or triangles}
\ \\ \\ 
The genus of quotient $\cH_q/G$ with $G\leq\PGU(3,q)$ was computed in \cite{MZNotFixing} whenever $G$ has no fixed points or triangles.
The equations of this section describe the genera of such quotients $\cH_q/G$.

\begin{theorem}{{\rm(}\cite{MZNotFixing}{\rm )}}
For any integer $\bar g$ provided by one of the Equations \eqref{hessian1} to \eqref{sottoPSUgenere}, there exists $G\leq{\rm PGU}(3,q)$ such that $g(\cH_q/G)=\bar g$ and $G$ has no fixed points or triangle.

Conversely, if $G\leq {\rm PGU}(3,q)$ has no fixed points or triangles, then $g(\cH_q/G)$ is given by one of the Equations \eqref{hessian1} to \eqref{sottoPSUgenere}.
\end{theorem}

\begin{equation}\label{hessian1}
\frac{q^2-34q+289}{432},\quad \frac{q^2-10q+25}{144},\quad \frac{q^2-10q+25}{72},
\end{equation}
where $G\cong\PGU(3,2)$, $G\cong\PSU(3,2)$, $G\cong SmallGroup(36,9)$, respectively.

\begin{equation}\label{pgl1}
\frac{q^2-16q+103-24\gamma-20\delta}{120},\quad\textrm{when}\quad p=5\quad\textrm{or}\quad 5\mid(q^2-1), \quad G\cong\mathbf{A}_5,
\end{equation}
$$\delta=\begin{cases} 2, \ if \ either \ p=3 \ or \ 3 \mid (q-1), \\ 0, \ if \ 3 \mid (q+1),\end{cases} \quad and \quad \gamma=\begin{cases} 0, \ if  \ 5 \mid (q+1), \\ 2, \ if \ p=5 \ or \ 5 \mid (q-1). \end{cases}$$
\begin{equation}\label{pgl2}
\frac{q^2-q-2-\Delta}{\bar q (\bar q +1)(\bar q -1)}+1,
\end{equation}
where $q=\bar{q}^h$, $\bar q\ne3$, $G\cong\PSL(2,\bar q)$, and
\begin{itemize}
\item $\Delta=+2(\bar q -2)( \bar q+1) +2 \frac{\bar q(\bar q+1)}{2} \bigg( \frac{\bar q-1}{2}-2\bigg)+\frac{\bar q ( \bar q+1)}{2}(q+1)+ \delta \frac{\bar q(\bar q-1)}{2} \bigg( \frac{\bar q+1}{2}-1\bigg),$ if $\bar q \equiv 1 \pmod 4$,
\item $\Delta=+ 2(\bar q -1)( \bar q+1) + 2\frac{\bar q(\bar q+1)}{2} \bigg( \frac{\bar q-1}{2}-1\bigg)+\frac{\bar q ( \bar q+1)}{2}(q+1)+\delta \frac{\bar q(\bar q-1)}{2} \bigg( \frac{\bar q+1}{2}-2\bigg)$, if $\bar q \equiv 3 \pmod 4$,
$${\rm with} \ \delta=\begin{cases} 2, \ if \ h \ is \ even, \\ 0, \ otherwise ;\end{cases}$$
\end{itemize}
\begin{equation}\label{pgl3}
\frac{q^2-q-2-\Delta}{2\bar q (\bar q +1)(\bar q -1)}+1,
\end{equation}
where $q=\bar{q}^h$, $\bar q\ne3$, $G\cong\PGL(2,\bar q)$, and
$$\Delta=2(\bar q-1)(\bar q+1)+\frac{\bar q ( \bar q+1)}{2}(q+1)+\frac{\bar q ( \bar q-1)}{2}(q+1)+2\frac{\bar q ( \bar q+1)}{2}(\bar q-1-2)+\delta \frac{\bar q ( \bar q-1)}{2}(\bar q+1-2)$$
and
$$\delta=\begin{cases} 2, \ if \ h \ is \ even, \\ 0, \ otherwise .\end{cases}$$

\begin{equation}\label{psl27}
\frac{q^2-22q+229-56\alpha-48\beta}{336},\quad\textrm{when}\quad p=7\quad\textrm{or}\quad \sqrt{-7}\notin\mathbb{F}_q
\end{equation}
where 
$$\alpha=\begin{cases} 0, \ if \ 3 \mid (q+1), \\ 2, \ otherwise; \end{cases} \beta=\begin{cases} 0, \ if \ 7 \mid (q+1), \\ 3, \ if \ 7 \mid (q^2-q+1), \\ 2, \ otherwise. \end{cases}$$

\begin{equation}\label{720}
\frac{q^2-10q+25}{72},\quad \frac{q^2-16q+55}{120}, \quad \frac{q^2-10q+25}{144},
\end{equation}
$$ \frac{q^2-46q+205}{720}, \quad \frac{q^2-46q+205}{1440},\quad\textrm{when}\quad q=5^n,\; n\;\textrm{is odd},$$
where $G\cong SmallGroup(36,9)$, $G\cong{\rm A}_5$, $G\cong\PSU(3,2)$, $G\cong{\rm A}_6$, $G\cong SmallGroup(720,765)$, respectively.

\begin{equation}\label{a6tutto}
\frac{q^2-46q+493-80\alpha-144\gamma}{720},\quad \frac{q^2-16q+103-20\alpha-24\gamma}{120},\quad
\frac{q^2-10q+25}{72},
\end{equation}
when either $p=3$ and $n$ is even, or $\sqrt{5}\in\mathbb{F}_q$ and $\mathbb{F}_q$ contains no primitive cube roots of unity, with 
$$\alpha=\begin{cases} 2, \ if \ p=3, \\ 0, \ otherwise; \end{cases} \  \gamma=\begin{cases} 0, \ if \ 5 \mid (q+1), \\ 2, \ otherwise,\end{cases}$$
and $G\cong{\rm A}_6$, $G\cong{\rm A}_5$, $G\cong SmallGroup(36,9)$, respectively.

\begin{equation}\label{a7grossi}
\frac{q^2-106q+2665-720\beta}{5040},\quad \frac{q^2-46q+205}{720},\quad \frac{q^2-22q+229-48\beta}{336},
\end{equation}
\begin{equation}\label{a7piccoli}
 \frac{q^2-26q+105}{240},\quad \frac{q^2-16q+55}{120},\quad \frac{q^2-10q+25}{72},
\end{equation}
where
$$ q=5^n,\;n\;\textrm{is odd},\quad\beta=\begin{cases} 0, & \textrm{if}\quad 7\mid(q+1), \\ 3, & \textrm{otherwise}, \end{cases} $$
and $G\cong{\rm A}_7$, $G\cong{\rm A}_6$, $G\cong\PSL(2,7)$, $G\cong{\rm A}_5\rtimes C_2$, $G\cong{\rm A}_5$, $G\cong SmallGroup(36,9)$, respectively.

\begin{equation}\label{sottoPGUgenere}
 1+\frac{q^2-q-2-\Delta}{2\bq^3(\bq^3+1)(\bq^2-1)},\quad \textrm{when}\quad \bar{q}=p^k,\; k\mid n,\; n/k\;\textrm{is odd},\; G\cong\PGU(3,p^k),
\end{equation}
where
$$ \Delta= (\bar q -1)(\bar q^3+1)\cdot(q+2) + (\bar q^3-\bar q)(\bar q^3+1)\cdot2 + \bar q(\bq^4-\bq^3+\bq^2)\cdot(q+1)$$
$$ + (\bq^2-\bq-2)\frac{(\bq^3+1)\bq^3}{2}\cdot2 + (\bq-1)\bq(\bq^3+1)\bq^2\cdot1 + (\bq^2-\bq)\frac{\bq^6+\bq^5-\bq^4-\bq^3}{3}\cdot\gamma, $$
with
$$ \gamma= \begin{cases} 3, & \textrm{if} \quad(\bq^2-\bq+1)\mid(q^2-q+1), \\
0, & \textrm{if} \quad(\bq^2-\bq+1)\mid(q+1).\\
\end{cases} $$
\begin{equation}\label{sottoPSUgenere}
\frac{3(q^2-q-2-\Delta)}{2{\bar q}^{3}({\bar q}^{2}-1)({\bar q}^{3}+1)}+1,\quad \textrm{when}\quad \bar{q}=p^k,\; k\mid n,\; n/k\;\textrm{is odd},\;3\mid(q+1),\; G\cong\PSU(3,p^k),
\end{equation}
where 
$$ \Delta= (\bar q -1)(\bar q^3+1)\cdot(q+2) + (\bar q^3-\bar q)(\bar q^3+1)\cdot2 + ((\bar q+1)/3-1)(\bq^4-\bq^3+\bq^2)\cdot(q+1)$$
$$ + ((\bar q^2-1)/3-(\bar q+1)/3)\frac{(\bar q^3+1)\bar q^3}{2}
\cdot2 + (\bq-1)((\bq+1)/3-1)(\bq^3+1)\bq^2 \cdot 1 +((\bar q^2-\bar q+1)/3-1)\frac{\bar q^6+\bar q^5-\bar q^4-\bar q^3}{3}\cdot\delta, $$
with 
$$\delta=\begin{cases}3, \ if \ (\bq^2-\bq+1)/3 \mid(q^2-q+1),\\ 0, \ if \ (\bq^2-\bq+1)/3 \mid(q+1). \end{cases}$$

\end{document}